\documentclass[a4paper,9pt,twoside]{article}
\usepackage{graphicx}
\usepackage[section]{placeins}
\usepackage[utf8]{inputenc}
\usepackage{amssymb}
\usepackage{amsmath}
\usepackage{amsthm}
\usepackage[linesnumbered,ruled,vlined]{algorithm2e} 
\usepackage{mathtools}
\usepackage[a4paper,text={150mm,255mm},centering,headsep=5mm,footskip=10mm]{geometry}
\usepackage[onehalfspacing]{setspace}
\usepackage[ngerman, english]{babel}
\usepackage{bbm}
\usepackage{grffile} 
\usepackage{verbatim}
\usepackage{caption}
\usepackage{subcaption}
\usepackage[dvipsnames]{xcolor}
\usepackage{gensymb}
\usepackage{tcolorbox}
\usepackage{enumitem}
\usepackage{colortbl}
\usepackage{fancyhdr}
\usepackage{url}
\usepackage[noadjust]{cite}
\usepackage{titlesec}
\usepackage{graphicx}
\usepackage{colortbl}
\usepackage{hyperref}
\usepackage{booktabs}
\usepackage{authblk}
\usepackage[symbol]{footmisc}

\tcbuselibrary{breakable}

\definecolor{mycolor}{rgb}{0,0,0}
\makeatletter
\newcommand{\greybox}[1]{%
  \setbox0=\hbox{#1}%
  \setlength{\@tempdima}{\dimexpr\wd0+13pt}%
  \begin{tcolorbox}[colframe=mycolor,breakable,boxrule=0.2pt,arc=4pt,colback=white!10,
      left=4pt,right=4pt,top=6pt,bottom=6pt,boxsep=0pt,width=
\textwidth] 
    #1
  \end{tcolorbox}
}

\newtheoremstyle{boldremark}
    {\dimexpr\topsep/2\relax} 
    {\dimexpr\topsep/2\relax} 
    {}          
    {}          
    {\bfseries} 
    {.}         
    {.5em}      
    {}          

\newtheorem{theorem}{Theorem}
\newtheorem{definition}[theorem]{Definition}
\newtheorem{corollary}[theorem]{Corollary}
\newtheorem{lemma}[theorem]{Lemma}
\newtheorem{proposition}[theorem]{Proposition}

\theoremstyle{boldremark}
\newtheorem{remark}[theorem]{Remark}
\newtheorem{example}[theorem]{Example}

\newcommand*{\R}{\mathbb{R}}
\newcommand*{\N}{\mathbb{N}}

\newcommand{\gammaerror}{\text{Barycentric coordinate error}}
\newcommand{\tgammaerror}{\text{Barycentric coordinate error }}
\newcommand{\truespaceerror}{\text{True space error}}
\newcommand{\ttruespaceerror}{\text{True space error }}

\newcommand{\datapoint}{x} 
\newcommand{\data}{\Xi} 
\newcommand{\Spoint}{\sigma} 
\newcommand{\Smatrix}{\Sigma} 
\newcommand{\gener}{\textbf{f}} 
\newcommand{\observ}{\textbf{g}} 
\newcommand{\memop}{\Psi} 
\newcommand{\dimD}{D} 
\newcommand{\timeT}{T} 
\newcommand{\clus}{K} 
\newcommand{\memory}{M} 

\DeclareMathOperator*{\argmin}{arg\,min}

\newcommand{\nw}[1]{\textcolor{blue!85!black}{\textbf{NW: }#1}}
\newcommand{\pk}[1]{\textcolor{red!85!black}{\textbf{PK: }#1}}

\numberwithin{equation}{section}
\numberwithin{theorem}{section}

\usepackage{geometry}
\definecolor{RoyalRed}{RGB}{157,16,45}
\titleformat{\chapter}[display]
  {\bfseries\LARGE}
  {\huge
  \chaptertitlename\hspace{0.1ex} \thechapter}{1pc}
  {{\titlerule[0pt]}\vspace{1pc}}

\titleformat{\section}[hang]{\bfseries\huge}{\bfseries\thesection}{1em}{}

\providecommand{\keywords}[1]
{
  \small	
  \textbf{\text{Keywords---}} #1
}

\usepackage[scaled]{helvet}
\usepackage{courier}

\begin{document}

\title{Data-driven modelling of nonlinear dynamics by barycentric coordinates and memory}
\author[1,2]{Niklas Wulkow}
\author[1]{P\'eter Koltai}
\author[2]{Vikram Sunkara}
\author[1,2]{Christof Schütte}
\affil[1]{Department of Mathematics and Computer Science, Freie Universität Berlin, Germany}
\affil[2]{Zuse Institute Berlin, Germany}
\maketitle

\begin{abstract}
    We present a numerical method to model dynamical systems from data. We use the recently introduced method Scalable Probabilistic Approximation (SPA) to project points from a Euclidean space to convex polytopes and represent these projected states of a system in new, lower-dimensional coordinates denoting their position in the polytope. We then introduce a specific nonlinear transformation to construct a model of the dynamics in the polytope and to transform back into the original state space. To overcome the potential loss of information from the projection to a lower-dimensional polytope, we use memory in the sense of the delay-embedding theorem of Takens. By construction, our method produces stable models. We illustrate the capacity of the method to reproduce even chaotic dynamics and attractors with multiple connected components on various examples.
\end{abstract}

\\\\
\textsc{MSC classes: 37M05, 37M10, 52B11}

\setcounter{tocdepth}{1}

\section{Introduction}

Learning dynamical models of real world systems from observations has become a critical tool in the natural sciences. For example, learning from femtosecond scale atomistic simulations and predicting  millisecond scale reaction kinetics is now a key aspect of novel drug discovery~\cite{bernetti}. These advances are made through mathematical properties, such as the identification of metastable states in protein folding~\cite{lane} or tipping points in the Earth's climate~\cite{wunderling}.

We can categorise the data-driven techniques for reconstructing dynamics into four partially overlapping families: kernel-based methods, regression methods, neural networks, and probabilistic models.  Neural networks methods determine a representation of the dynamics in compositions of affine transformations and nonlinear activation functions~\cite{teng,vlachas,kutzNNbook,champion,stuart}. Kernel-based methods approximate the dynamics using suitable distance measures, the kernels, in the data space~\cite{gilani,scherer}. Regression methods determine a linear mapping between nonlinear, often polynomial, transformations of the data and future states of the dynamics (or nonlinear observables thereof), such as DMD~\cite{dmd}, EDMD~\cite{liEDMD} or SINDy~\cite{sindy,sindyc}. Lastly, probabilistic models use discretisations of the data space and describe the dynamics by transitional probabilities between the discrete states, such as Markov State Models (MSM)~\cite{husic} and the recently introduced method Scalable Probabilistic Approximation (SPA)~\cite{spaPaper}. 

Despite current success of these methods, tackling high-dimensional nonlinear dynamics is still a challenge. This is, because neural networks require an abundance of parameters whose estimation is costly (e.g., due to a non-convex optimization problem), kernel-based methods require a suitable choice of the kernel~\cite{hamzi}, regression methods need the right choice of transformations, making a suitable basis desirable (for instance by restricting the parametrization to collective variables or reaction coordinates~\cite{bittracher}), and can result in ill-conditioned problems~\cite{freund}. Lastly, probabilistic models are more stable and do not require strong intuition about the problem, however often suffer from the curse of dimensionality. Also, they typically only allow conclusions on the level of their discretisation. This drawback has been partially remedied by SPA, which chooses a finite set of landmarks in the data space and represents the data points by stochastic vectors which encode the strength of affiliation with the landmarks. This is in contrast to the binary affiliation with boxes as in MSMs and results in a lower loss between original data and SPA representations. Furthermore, it was shown in \cite{spaPaper} that this leads to high predictive accuracy for complex examples of both high and low dimension in problems from climate science and biology while coming at the same computational complexity as the cheap K-means clustering method. 

Based on the chosen number of landmarks, $K,$ and their positions in the state space, SPA estimates linear forward models based on the law of total probability. More precisely, SPA was invented for modeling a process $Y(t)$ in terms of another process $X(t)$ by finding linear relations between finite probabilistic representations of $X$ and~$Y$. It provides  a low-cost algorithm, allowing for simultaneous data-driven optimal discretisation, feature selection, and prediction, especially if $X$ and $Y$ are both high-dimensional and stochastic. In this article, we show how to extend SPA for usage in a slightly different setting, namely for forecasting a single process~$X$. 

As discussed in various publications (see, e.g.,~\cite{lusch,sindy,champion,liEDMD}), approximations of the linear but infinite-dimensional Koopman operator can generate sound linear approximations of nonlinear dynamics. However, they require the transformation into a suitable space where the dynamics can be approximated by a finite-dimensional linear propagator. To faithfully reconstruct long-term dynamics, we seek in the present work a nonlinear approximation of the dynamics using the cost-effective SPA representation of our high-dimensional data. As we will discuss in Section~\ref{sec:memory}, this method also admits a dual view, as it can be seen as a nonlinear reconstruction of dynamics and a linear reconstruction of the transfer operator.

Specifically, we propose not to choose a large number $K$ of discrete landmarks and construct a linear model, but rather use a small $K$ and construct a nonlinear model.
In particular, we use the outer-products of the memory terms of the SPA representation, later called \textit{path affiliations}, as our nonlinear basis functions. Contrary to convention, we choose a fixed basis over our memory terms, which preserves the probabilistic interpretation of the SPA representations. The dynamics that we learn on the path affiliations are linear and are represented by a stochastic matrix. This has the advantage that the emerging dynamics are \emph{stable} by construction---a widely desirable property~\cite{fan2021learning,BBKLSH21}.
We demonstrate that this choice of basis functions in connection with the SPA representation and memory can construct good approximations of high-dimensional nonlinear dynamics from data.


This paper is divided into three parts. 
In Section~\ref{sec:spa}, we introduce the SPA method and deconstruct the different sub-problems. We introduce a new interpretation of SPA representations of data as projections onto a polytope and theoretically explore SPA in the context of dynamical systems. 
In Section~\ref{sec:memory}, we introduce an extension to SPA based on memory which we refer to as mSPA (memory SPA). We then prove using Takens' Theorem that we can construct topologically equivalent dynamics along the path affiliations. In Section~\ref{sec:scheme}, we show how to use mSPA to learn the dynamics on the SPA representation (the Learning) and then to map it back to the original data space (the Lifting). We compute this on three examples, one simple, one high-dimensional (Kuramoto--Sivashinsky PDE), and lastly on a chaotic system (Lorenz-96).

Figure~\ref{fig:mSpa_diagram_reduced} provides an illustration of the scheme that is introduced in this article. The objects on the arrows will be defined in due course.
\begin{figure}[ht]
\centering
\includegraphics[width = 0.7\textwidth]{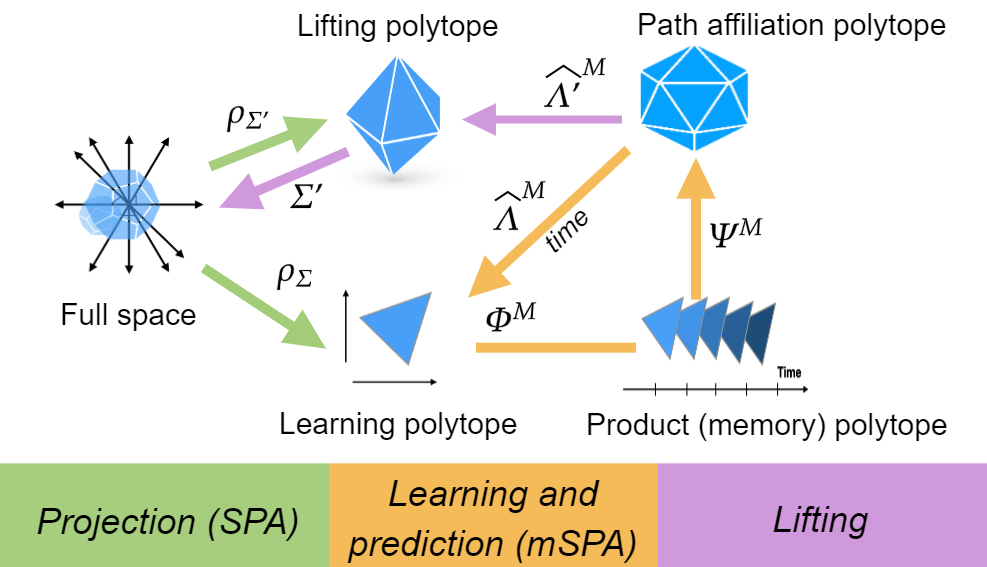}
\caption{Concept of the Learning and Lifting approach introduced in this article. SPA is used to project onto a polytope, dynamics are propagated on the polytope using the path affiliations and states on the polytope are lifted back to the original data space, using a lifting polytope as an intermediate step.}
\label{fig:mSpa_diagram_reduced}
\end{figure}

\section{Scalable Probabilistic Approximation (SPA)}
\label{sec:spa}

Data-driven methods for learning dynamics entail two fundamental steps: the first step being \emph{feature selection} and the second step being \emph{dynamics reconstruction}. 
The first step entails projecting the data into a known \emph{feature space} to help delineate the underlying dynamics of the data.
Then, in the second step, a \emph{propagator} over the projected data in the feature space is inferred which recapitulates the dynamics of the full data.
Choosing the right feature spaces is a key bottleneck for most datasets. This is particularly so for high-dimensional datasets; where feature spaces which capture the right geometry of the data, in practise, are found to be sensitive to noise and ill-conditioned when inferring the propagator. 
One method which balances the choice of an interpretable feature space and the stability of the propagator is the \emph{Scalable Probabilistic Approximation} (SPA) method~\cite{spaPaper,eSPAIllia}.

More precisely, SPA was invented for modeling a process $Y(t)$ in terms of another process $X(t)$ by finding linear relations between finite probabilistic representations of $X$ and $Y$, in the sense that the optimal matrix $\Lambda$ of conditional probabilities is computed for which the two finite dimensional probabilistic representations $\Gamma^X$ and $\Gamma^Y$ are linearly connected, $\Gamma^Y=\Lambda\Gamma^X$. SPA was designed for (i) finding optimal representations $\Gamma$ and (ii) simultaneously identify the optimal linear model $\Lambda$, especially for cases where $X$ and $Y$ are both high-dimensional and stochastic. In this article, SPA is used in a slightly different setting, namely for forecasting a single process~$X$. We will show how to extend SPA by incorporating memory for allowing for long-term predictions of $X$. To this end, we will present SPA in a form where the two steps (i) and (ii) are presented individually.

\subsection{The SPA method revisited}
Like other data driven methods, the principle behind SPA can be divided into two steps: the first to project the data onto a convex polytope and secondly to infer a linear propagator over the said polytope.
The first of the two steps, which we refer to as ``SPA I", is formulated as follows: for a sequence of $\timeT \in \N$ observations of $(\dimD \in \N)$-dimensional data points, $\data := [\datapoint_1\vert \cdots\vert\datapoint_\timeT],$ the $\clus \in \N$ vertex convex polytope to project the data onto is determined by solving for the matrices $\Smatrix \in \R^{\dimD \times \clus}$ and $\Gamma \in \R^{\clus \times \timeT}$ which satisfy:
\begin{equation*}
\tag{SPA I}
\begin{split}
&\hspace{2cm}[\Smatrix,\Gamma] = \argmin_{\Smatrix^*,\Gamma^*} \Vert \data - \Smatrix^*\Gamma^* \Vert_F,\\
\text{subject to,}& \\
& \Smatrix := [\Spoint_1\vert \cdots \vert \Spoint_K] \in \R^{\dimD\times K}, \quad \Gamma := [\gamma_1\vert \cdots\vert \gamma_\timeT] \in \R^{\clus\times \timeT},  \\
\text{with}&  \\
&\sum\limits_{k=1}^\clus \gamma_{t}^k = 1 \text{ and }  \gamma_{t}^k \geq 0 \text{ for all } t \in \{1,\ldots, T\}\text{ and }k \in \{1,\ldots, K\}. \\
\end{split}
\label{eq:SPA1}
\end{equation*}

We can deduce from the above formulation that $\sigma_\bullet,$ the columns of the matrix $\Smatrix,$ represent a new coordinate system consisting of $\clus$ basis points in $\R^\dimD$. Furthermore, $\gamma_\bullet,$ the column vectors of $\Gamma,$ are simply an orthogonal projection of data points $x_\bullet$ onto the convex polytope formed by the span of $\{\Spoint_1, \ldots , \Spoint_K \}.$ For each vertex $k \in \{1,\ldots, \clus \}$, the value of $\gamma_{\bullet}^k$ can be interpreted as a closeness/affiliation measure of the data point $\datapoint_\bullet$ to the vertex $\Spoint_k$. For brevity, here on in, we will refer to the ``convex polytope" spanned by vectors $\sigma_\bullet$ as simply a ``polytope" and denote it by~$\mathcal{M}_\Smatrix.$ Further, if $\mathcal{M}_\Sigma$ is a simplex, there exists a bijective map between the so-called \textit{barycentric coordinates} $\gamma$ to points $\Sigma \gamma \in \mathcal{M}_{\Sigma}$ so that we use $\mathcal{M}_{\Sigma}$ interchangeably for $conv(\Sigma)$ and the polytope of the barycentric coordinates since each point in $\mathcal{M}_{\Sigma}$ has a unique representation in barycentric coordinates~\cite{guessab}. We will refer to $\R^\dimD$ as the \textit{full space} and the dynamics in this space as the \textit{full system}. We show in Appendix~\ref{sec:App_rhosigma} that if $K\leq \dimD$, we can always assume $\mathcal{M}_{\Sigma}$ to be a simplex. If $\mathcal{M}_{\Sigma}$ is not a simplex, the representation of a point with barycentric coordinates might not be unique. We address this case later on in this section.

\begin{remark}
The idea of projecting to a convex polytope derived from the data has emerged in different sciences under different names, e.g.\ Archetypal Analysis~\cite{cutler} in computer science and PCCA+~\cite{weberpcca} in chemistry. The key motivation being that the vertexes have clear interpretabilty given they can be seen as ``reference"/``landmark" points in the data space. Furthermore, all data points already in the polytope have no projection error. However, on the downside, this is a linear decomposition of the data, meaning, nonlinear dynamics in the data-space will still be nonlinear on the polytope. In Section~\ref{sec:mSPA} we introduce a specific nonlinear transformation of these new coordinates that helps to capture nonlinearity.

\end{remark}

After projecting the data points from $\R^\dimD$ into the polytope $\mathcal{M}_\Sigma,$ the second step of SPA is to find a column-stochastic matrix $\Lambda \in \R^{\clus \times \clus} $ which propagates the projected points in the polytope. We refer to this step as ``SPA II", and formulate it as follows: let the matrices $\Smatrix \in \R^{\dimD \times \clus}$ and $\Gamma \in \R^{\clus \times \timeT}$ be the solution to \eqref{eq:SPA1} for a dataset $\data := [\datapoint_1\vert \cdots\vert\datapoint_\timeT],$ then find $\Lambda$ such that,
\begin{equation*}
\tag{SPA II}
\begin{split}
\Lambda &= \argmin_{\Lambda^*\in \R^{K\times K}} \Vert [\gamma_2 \vert \cdots \vert\gamma_\timeT] - \Lambda^* [\gamma_1 \vert \cdots\vert\gamma_{\timeT-1}] \Vert_F, \\
\text{satisfying,} & \\
& \Lambda \geq 0 \text{ and } \sum_{k=1}^K \Lambda_{k,\bullet} = 1. 
\end{split}
\label{eq:SPA2}
\end{equation*}

The propagator $\Lambda$ gives a linear approximation of the evolution of the full system on the polytope $\mathcal{M}_\Sigma.$ Furthermore, the propagator $\Lambda$ being a column-stochastic matrix guarantees that the image of $\mathcal{M}_{\Sigma}$ is again inside the polytope $\mathcal{M}_{\Sigma}$. In practice, the stability of reconstructions of non-linear dynamics plays an important role in applications (see, e.g.,~\cite{fan2021learning}), and we note that the method~\eqref{eq:SPA2} satisfies it by construction.
Though the full system might be nonlinear, using the SPA method reduces the learning of the dynamics on the data-space to learning a linear system on a convex polytope. To date, the literature on the translation of the dynamics after projecting into a polytope is sadly vacuous. For this reason, we will briefly formalise and give commentary on the SPA method in the context of dynamical systems in the following section.

\begin{remark}
We have to clarify that in the original SPA paper~\cite{spaPaper}, Gerber et al.\ refer to \eqref{eq:SPA1} as the SPA problem and refer to other literature for solving \eqref{eq:SPA2} (e.g.,~\cite{huisinga,olsson,rubin}). Hence, the notation of studying the projection and the propagator separately, \eqref{eq:SPA1} and \eqref{eq:SPA2}, is our interpretation. Furthermore, Gerber et al.\ prove that the solution to their one step approach is equivalent to the two step (see Theorem~2 in~\cite{spaPaper}). Hence, computing the SPA problem separately or together only differs numerically and not analytically. 
\end{remark}

\subsection{SPA in the context of dynamical systems (Part I)}
\label{sec:SPAdynSystems}
In this section we formalise the \eqref{eq:SPA1} map of a dynamical system into a polytope. With this map, we can derive a sufficient condition which can help translate the dynamical system in the full data-space into a new dynamical system in the polytope.

\begin{definition}
For $x \in \mathcal{M}$ and $\gamma$ a $K$-dimensional barycentric coordinate, we define the mapping of $x$ into a $K$-vertexed polytope characterised by $\Sigma$ as follows:

\textbf{Case 1: } for $K \leq D$:
\begin{equation}
\begin{split}
    &\hspace{1.6cm}\rho_\Sigma(\datapoint) := \argmin\limits_{\gamma^*}\Vert \datapoint - \Smatrix\gamma^*\Vert_2, \hspace{1cm} \text{s.t. } \gamma^*_\bullet \geq 0, \quad \Vert \gamma^* \Vert_1 = 1.
    \end{split}
    \label{eq:rhoK<D}
\end{equation}

\textbf{Case 2: } for $K > D$ : 
\begin{equation}
\begin{split}
    &\hspace{1.4cm}\rho_\Sigma(\datapoint,\gamma) := \argmin\limits_{\gamma^*} \Vert \gamma - \gamma^* \Vert_2 \\ & 
    \text{ s.t. }\gamma^* = \argmin\limits_{\gamma'}\Vert \datapoint - \Smatrix\gamma'\Vert_2
    \text{ with } \gamma'_\bullet,\gamma^*_\bullet  \geq 0 \text{ and } \Vert \gamma' \Vert_1, \Vert \gamma^* \Vert_1 = 1.
    \end{split}
    \label{eq:rhoK>D}
\end{equation}
\end{definition}

Case 2, $K > D,$ is interesting because generally points can be represented using different barycentric coordinates. Therefore, there is no natural choice for a function that maps the data points ($\datapoint$) to points in the polytope ($\gamma$). For this reason, we define the mapping $\rho_\Sigma$ to be dependent not only on $\datapoint$ but also a reference coordinate $\gamma.$ The intuition is that we want the barycentric coordinate of a point to minimize the distance between $\datapoint_t$ and $\Smatrix\gamma_t$ while being closest to the reference coordinate among all those minimizers. This now makes the mapping into the polytope well-defined. This is also very natural if $\gener$ represents time-stepping of a continuous-time dynamical system with small time step.

For Case 1, $K \leq D,$ we can always assume that the \eqref{eq:SPA1} solution is a $(K-1)$-dimensional simplex (a simplex is the convex hull of $K$ affine-invariant points; see Lemma~\ref{pro:rhosimplex} in Appendix~\ref{sec:App_rhosigma}). Furthermore, the map $\rho$ is also an orthogonal projection (see Lemma~\ref{lem:rhowelldefined} in \ref{sec:App_rhosigma}).

Let us consider a discrete dynamical system on a compact manifold $\mathcal{M}$ of the form,
\begin{equation*}
   \datapoint_t = \gener(\datapoint_{t-1}),
\end{equation*}
where $x_\bullet \in \mathcal{M}$ and $f:\mathcal{M} \rightarrow \mathcal{M}.$

Let us now restrict ourselves to the case $K > \dimD$. In this case, we can easily construct the following dynamical system for $\gamma_{t}$ by choosing $\gamma_{t-1}$ as reference coordinate in each time step. Then
\begin{align}
    \gamma_t &= \rho_\Sigma(\datapoint_t,\gamma_{t-1}),\\
    \intertext{we substitute $\datapoint_t$ by $\gener(\datapoint_{t-1})$,}
    &= \rho_\Sigma(\gener(\datapoint_{t-1}),\gamma_{t-1}).\\
    \intertext{Furthermore, by the representation of $\datapoint_{t-1}$ dependent on $\gamma_{t-1}$,}
    &= \rho_\Sigma(\gener(\Smatrix\gamma_{t-1}),\gamma_{t-1}),\\
    \intertext{and obtain a function solely depending on $\gamma_{t-1}$,}
    &=: \textbf{v}(\gamma_{t-1}).
    \label{eq:gammaDynSys}
\end{align}
The function $\textbf{v}: \mathcal{M}_\Sigma \rightarrow \mathcal{M}_\Sigma$ now defines the dynamical system dependent on the $\gamma_\bullet$.

Hence, in \eqref{eq:SPA2}, we are approximating the function $\textbf{v} $ by a linear mapping given by the matrix $\Lambda$, that is, $\gamma_{t} \approx \Lambda \gamma_{t-1}.$ Furthermore, the long-term behaviour, for $k$ future steps, is thus approximated by
\begin{equation}
    \gamma_{t+k} = \textbf{v}^{k} (\gamma_{t}) \approx \Lambda^{k} \gamma_{t}.
    \label{eq:SPAII_dynamicsApproximation}
\end{equation}
The propagator $\Lambda$ being a column-stochastic matrix keeps the predictions inside the polytope. However, $\Lambda$ is still only a linear approximation of $\textbf{v},$ making it imprecise in the long term when $\textbf{v}$ is nonlinear (see also Appendix~\ref{sec:App_SPAfixedpoints}). In numerous publications, e.g., recently~\cite{lusch}, the representation of nonlinear dynamics using the linear Koopman operator is discussed. However, this operator is infinite-dimensional.

In Section~\ref{sec:memory}, we will introduce a memory-based method that takes the projection error into account and propagates states inside the polytope in a nonlinear way. Furthermore, will show how we can use memory to also extend the dynamics into a polytope for the case when $K < D.$
\section{Memory SPA}
\label{sec:memory}

In this section, we will present an extension to SPA to better approximate nonlinear dynamics while staying in the polytope setting. The key motivation for the extension comes from the use of memory and the application of Takens' Theorem. In order to introduce memory, for a dynamical system on a smooth compact $\dimD$-dimensional manifold $\mathcal{M}$, with a generator function $\gener:\mathcal{M}\rightarrow \mathcal{M}$ that we assume to be invertible and an observable function $\observ:\mathcal{M}\rightarrow \R,$  we define a memory $\memory$-\textit{delay-coordinate map} ($\memory\in \mathbb{N}$) of the points in the manifold $\mathcal{M}$ as the function: 
\begin{equation}
    \Phi^M(\datapoint) := ( \observ(\datapoint), \observ(\gener^{-1}(\datapoint)),\dots,\observ(\gener^{-\memory+1}(\datapoint))).
    \label{eq:delayembedding}
\end{equation}

In the above scenario if $\memory > 2 \dimD,$ we know from Takens theorem, that the $\memory$-delay-coordinate map, $\Phi^\memory,$ is an embedding~\cite{takens} (see Appendix~\ref{sec:App_Takens} for the formal construct). Furthermore, Takens' Theorem asserts the existence of a \textit{topologically equivalent} dynamical system on the image of $\Phi^M$ given by,
\begin{equation}
    \Phi^\memory(\datapoint_{t+1}) = \Phi^\memory \circ \gener \circ  (\Phi^M)^{-1}(\observ(\datapoint_t),\dots,\observ(\datapoint_{t-\memory+1})),
    \label{eq:Takensimplies}
\end{equation}
where $(\Phi^M)^{-1}(\observ(\datapoint_t),\dots,\observ(\datapoint_{t-\memory})) = \datapoint_t$. Furthermore, we can use the inverse map to reformulate the future state of the dynamical system, $\datapoint_{t+1} = \gener(\datapoint_t),$ as a function of the $M$ memory states of the observable,
\begin{equation}
\datapoint_{t+1} = \gener((\Phi^\memory)^{-1}(\observ(\datapoint_t),\dots,\observ(\datapoint_{t-\memory+1}))).
\label{eq:TakensimpliesGammaToX}
\end{equation}

We now harness this property by choosing our observable $\observ$ as the projection of the data $\datapoint$ onto polytope $\mathcal{M}_{\Sigma}$. In Takens' original formulation, the observable $\observ$ is scalar-valued, whereas, the projection onto the polytope is a $K$-dimensional map. However, the theorem has been extended to hold for multivariate observables, e.g.,~\cite{sauer}. Hence, from here on in, we consider our observables to be multivariate.

\subsection{Memory SPA: a path affiliation approximation}
\label{sec:mSPA}
We now propose a specific observable that fits the construction of \eqref{eq:SPA2} particularly well to give a new numerical method. Let $\gamma_t$ denote the $K$-dimensional barycentric coordinates of a point $\datapoint_t \in \mathcal{M}_{\Sigma}$ and let $\gamma_t^i$ denote its $i$th entry. For a fixed memory depth $\memory$, we define as the \textit{path affiliations} the function
\begin{align}
\label{eq:PSI_def}
\memop^M(\gamma_{t-1},\dots,\gamma_{t-\memory})_i &:=  \gamma_{t-1}^{i_1}\cdot \gamma_{t-2}^{i_2} \cdots \gamma_{t-\memory}^{i_\memory}\\
J(i) &= [i_1,\dots,i_\memory]
\end{align}
$\memop^\memory(\gamma_{t-1},\dots,\gamma_{t-\memory})$ consists of all products of combinations of entries from $\gamma_{t-1}$ to $\gamma_{t-\memory}$. $J:\lbrace 1,\dots,\clus^\memory \rbrace \rightarrow \lbrace 1,\dots, \clus \rbrace^\memory$ denotes a function that brings all $\clus^\memory$ different combinations of the indices $1,\dots,\clus$ into a certain order (without loss of generality). 
For example, with $\clus = 2$ and $\memory=2$,
\begin{equation*}
\memop^2(\gamma_t,\gamma_{t-1}) = \begin{pmatrix}
\gamma_{t}^1\cdot\gamma_{t-1}^1\\
\gamma_{t}^1\cdot\gamma_{t-1}^2\\
\gamma_{t}^2\cdot\gamma_{t-1}^1\\
\gamma_{t}^2\cdot\gamma_{t-1}^2
\end{pmatrix}.
\end{equation*}
We can also write the path affiliation function \eqref{eq:PSI_def} in the form of a sequence of outer products,
\begin{align}
        \memop^\memory(\gamma_{t-1},\cdots,\gamma_{t-\memory}) &:= \gamma_{t-1} \otimes \dots \otimes \gamma_{t-\memory},\\ 
        u \otimes v &:= \text{vec}(uv^T).
    \label{eq:PSI_def_outerproduct}
\end{align}

We recall that, for a time $t,$ the barycentric coordinates $\gamma_t^\bullet$ denote the closeness/affiliation of their respective data points $\datapoint_t$ to the vertices $\sigma_{\bullet}$ from the solution of the \eqref{eq:SPA1} problem. The function $\memop^\memory$ simply takes the product of the barycentric coordinates of a sequence of time points, thereby, extending the affiliation to include all vertices through time. Furthermore, the path affiliation function is surjective onto a $K^\memory$-vertexed polytope (see Appendix~\ref{sec:App_psiimage} for the definition of that polytope) and its image is also a simplex, since path affiliations are stochastic vectors (Appendix~\ref{sec:App_pathaffStoch}).
Additionally, $\memop^\memory$ is injective if restricted to the polytope:
\begin{proposition}
\label{pro:psiinjective}
The path affiliation function $\memop$ as defined in Eq.~\eqref{eq:PSI_def} is injective.
\end{proposition}
\begin{proof}
See Appendix~\ref{sec:App_psiinjective}
\end{proof}

Now, keeping in line with SPA, we want an operator which can map points from the image of the path affiliation function back into the original polytope $\mathcal{M}_\Sigma.$ That is, we want a column-stochastic matrix $\hat{\Lambda}^M \in \R^{K\times K^{\memory}},$ such that, 
\begin{equation}
    \gamma_t = \hat{\Lambda}^\memory\, \memop^\memory(\gamma_{t-1},\dots,\gamma_{t-\memory}).
    \label{eq:mSPAsystem}
\end{equation}
Given the path affiliation function $\memop^\memory$ is nonlinear, we hope that with the combination of $\memop^\memory$ and $\hat{\Lambda}^\memory$ we can accurately approximate nonlinear systems. Given that $\hat{\Lambda}^\memory$ is column-stochastic, predictions are still restricted to the \eqref{eq:SPA1} polytope since we multiply a column-stochastic matrix with a stochastic vector.

Now we will formulate our extension to the \eqref{eq:SPA2} problem. Let us denote by $[\Sigma,\Gamma]$ the solution to \eqref{eq:SPA1} for data points $\datapoint_1,\dots,\datapoint_T \in \R^\dimD$ with vertices $\sigma_1,\dots,\sigma_K$. Let us define $\psi^\memory_{t-1} := \memop^\memory(\gamma_{t-1},\dots,\gamma_{t-\memory})$. Then,
\begin{equation}
\tag{mSPA}
\begin{split}
\hat{\Lambda}^\memory &:= \argmin_{\hat{\Lambda}^* \in \R^{K\times K^{\memory}}} \left\Vert [\gamma_{\memory+1} \vert \cdots \vert\gamma_\timeT] - \hat{\Lambda}^* [\psi^\memory_{\memory} \vert \cdots\vert\psi^\memory_{\timeT-1}] \right\Vert_F, \\
\text{satisfying,} & \\
& \hat{\Lambda}^\memory \geq 0 \text{ and } \sum_{k=1}^{K} \hat{\Lambda}^\memory_{k,\bullet} = 1. 
\end{split}
\label{eq:mSPA2}
\end{equation}
We refer to the matrix $\hat{\Lambda}^\memory$ as the \emph{mSPA propagator}. Furthermore, the mSPA propagator $\hat{\Lambda}^\memory$ maps path affiliations back onto the \eqref{eq:SPA1} polytope.

Note that, Eq.~\eqref{eq:mSPAsystem} is not a dynamical system in the true sense since it is not closed. In order to define a closed dynamical system, we can simply augment the propagator by defining $\hat{\gamma}_t := [\gamma_t^T,\dots,\gamma_{t-\memory+1}^T]^T$ and write
\begin{equation}
    \hat{\gamma}_t = \Theta^\memory \memop^\memory(\hat{\gamma}_{t-1}) = \Theta^\memory \psi^\memory_{t-1},
    \label{eq:closedmemorydynamics}
\end{equation}
where $ \psi^\memory_{t-1} := \memop^\memory(\gamma_{t-1},\dots,\gamma_{t-\memory})$ and $\Theta^\memory = \left[ \begin{smallmatrix}
  \hat{\Lambda}^\memory \\ E \end{smallmatrix} \right] \in \R^{K^\memory \times K^\memory}$. Here $E$ is a binary matrix that retains $\gamma_{t-1},\dots,\gamma_{t-\memory+1}$ from $\psi^\memory_{t-1}$ (see Appendix~\ref{sec:App_choiceE} for details).
Equivalently (by Appendices~\ref{sec:App_psiimage} and~\ref{sec:App_psiinjective}), we can define a dynamical system on the path affiliations by
\begin{equation}
    \psi^\memory_t = \memop^\memory (\Theta^\memory \psi^\memory_{t-1}).
    \label{eq:closedpsidynamics}
\end{equation}
With this, the mSPA propagator, together with the path affiliation function $\memop^\memory$, progresses the dynamics in the $K^\memory$-vertexed polytope of path affiliations.
Actually, even more can be said about~\eqref{eq:closedpsidynamics}. As we show in Appendix~\ref{sec:App_quad}, the dynamics $\memop^\memory \circ \Theta^\memory$ is component-wise a quadratic polynomial in the components of~$\psi^\memory$.

\begin{example}
\label{exa:overlap}
\textbf{Observed dynamics of the Lorenz-96 system }
Let us consider the following 5-dimensional Lorenz-96-system, a coordinate-wise hierarchical continuous-time dynamical system, which is described by, 
\begin{equation*}
    \dot{x}^{(k)} = (x^{(k+1)} - x^{(k-2)})x^{(k-1)} - x^{(k)} + F \quad \text{ for } k = 1,\dots 5,
\end{equation*}
with initial condition $\datapoint_0 = (3.8,3.8,3.8,3.8,3.8001)$ and $F = 3.8.$

We now apply mSPA to a two-dimensional observation -- the $x^{(1)}$- and $x^{(4)}$-coordinates -- of the Lorenz-96-system. We used \ref{eq:SPA1} with $K=3$ to find a representation for the points and applied mSPA with different memory depths to estimate $\hat{\Lambda}$. We then made forward predictions using Eq.~\eqref{eq:mSPAsystem}. The projection onto the $\datapoint^{(1)}$- and $\datapoint^{(4)}$-coordinates gives a figure-eight shape in Euclidean space and also for the barycentric coordinates. As can be seen in Figure~\ref{fig:Lorenz96_x1x4_Fullplot}, this shape of the attractor was reconstructed by a long trajectory (1000 steps of length 0.1 -- note that the length of one period is roughly $2.5$ so that the 20-step error in Figure~\ref{fig:Lorenz96_x1x4_Fullplot} pertains to a little less than one period) only with $\memory = 6$. Of course, in this example we had not reduced the dimension of observed points any further by applying \ref{eq:SPA1} with $K = 3$ to two-dimensional data from the observable. This example only serves to illustrate the capacity of mSPA in application to barycentric coordinates. Since the $\datapoint^{(1)}$- and $\datapoint^{(4)}$-coordinates denote an observable of the full, 5-dimensional system, according to Takens' Theorem memory was necessary to describe the dynamics. This can also be seen since the overlap could not be recreated otherwise. For a memoryless prediction, the next step from the point of overlap would not be well-defined.

Note that, since the full Lorenz-96-system is 5-dimensional, by Takens, we should need at most 11 coordinates as entries of the delay-coordinate map. Since for $K=3$ we worked with 2 independent coordinates in $\gamma$, we naturally found that for $\memory = 6$, the delay-coordinate map gives a 12-dimensional map which was sufficient to contain the dynamical information of the true system.

\begin{figure}[ht!]
\centering
\includegraphics[width=0.9\textwidth]{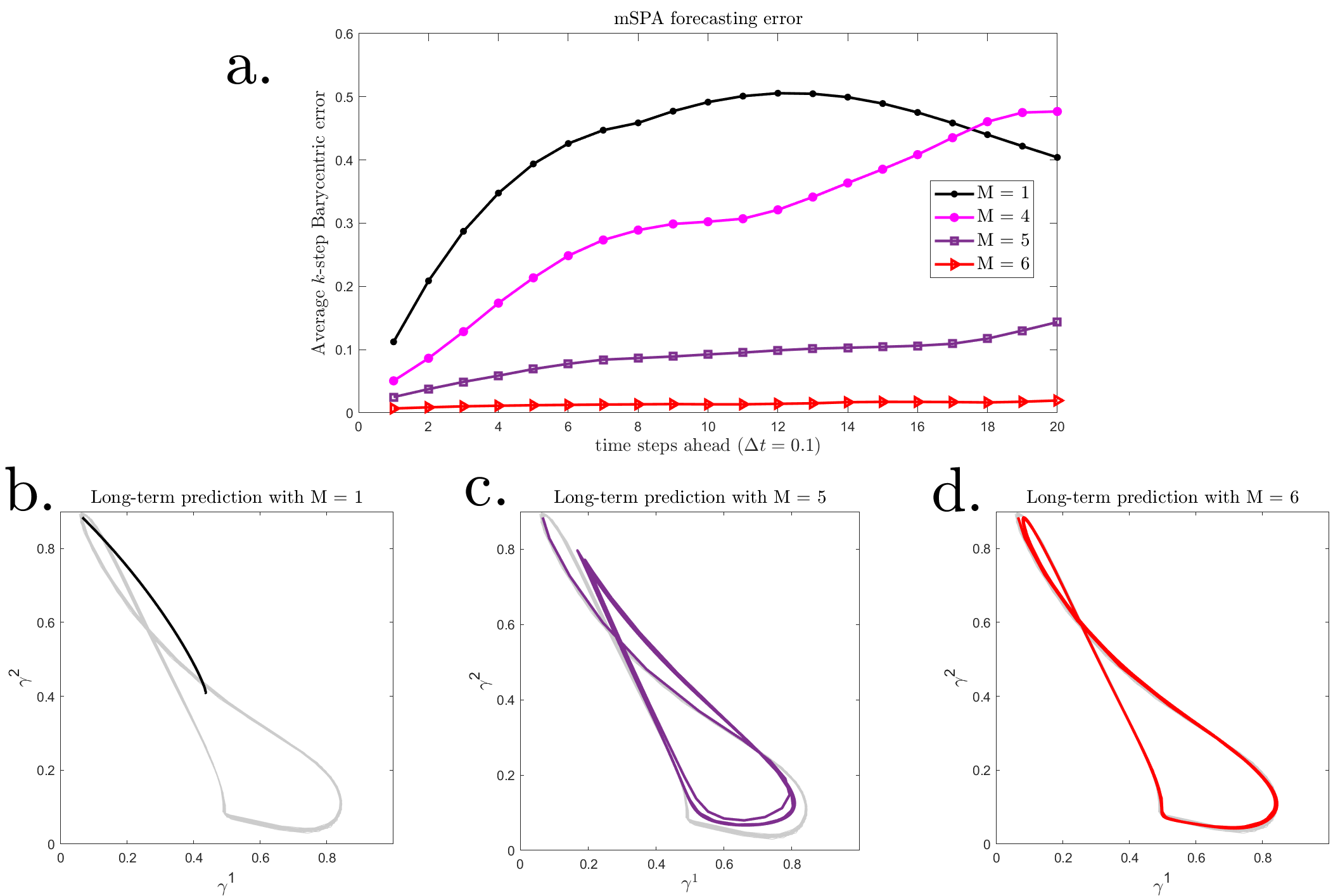}
\caption{Long-term prediction with mSPA for the Lorenz-96 system. \textbf{a.} shows the forecasting error up to 20 steps of length $\Delta t = 0.1$. Higher memory improves the prediction from the first forward-step onward. For the long-term prediction (\textbf{b.}--\textbf{d.}), only with $\memory=6$ the Lorenz-96 attractor for the $\datapoint^{(1)}$- and $\datapoint^{(4)}$-coordinates is reconstructed.}
\label{fig:Lorenz96_x1x4_Fullplot}
\end{figure}
\end{example}
\FloatBarrier

We close this section by showing that mSPA is a generalization of SPA II. For any given $K$, any dynamical system generated by \eqref{eq:SPA2} can be generated by mSPA:
\begin{proposition}
For every column-stochastic matrix $\Lambda \in \R^{K\times K}$ and every $\memory \geq 1$, there exists a column-stochastic matrix $\hat{\Lambda}^\memory \in \R^{K\times K^\memory}$ so that $\Lambda \gamma_{t-1} = \hat{\Lambda}^\memory \memop^\memory(\gamma_{t-1},\dots,\gamma_{t-\memory})$ for all $\gamma_{t-1},\dots,\gamma_{t-\memory}$.
\label{pro:mSPAvsSPA2}
\end{proposition}
\begin{proof}
See Appendix~\ref{sec:AppProofLemma}.
\end{proof}
From this follows a similar result about the fit of the corresponding minimization problem:
\begin{corollary}
The minimal residual in the \eqref{eq:mSPA2} problem is less than or equal to the minimal residual in the \eqref{eq:SPA2} problem.
\label{lem:mSPAresidual}
\end{corollary}

\FloatBarrier
\subsection{SPA in the context of dynamical systems (Part II)}


The precise reason why we use memory terms in the construction of the method has not been discussed yet. Recall that \ref{eq:SPA1} serves as a dimension reduction method so that typically $K < \dimD$, even $K \ll \dimD$. Also, recall that in Eq.~\eqref{eq:gammaDynSys}, we  formulate the dynamics of $\gamma$ by expressing $\datapoint_{t-1}$ as a function of $\gamma_{t-1}$ using an injection from $\datapoint$ to $\gamma$. However, if $K\leq D$ then generally the polytope does not contain all data points. In other words, there exists a point $\datapoint\notin \mathcal{M}_{\Smatrix}$ so that there is no $\gamma$ that fulfills $\datapoint = \Sigma \gamma$. As a consequence, multiple points from outside the polytope can be mapped onto the same point on the polytope so there is no injection onto $\mathcal{M}_\Sigma$. 

For this reason, we will now utilize a variant of Takens' Theorem that allows us to formulate well-defined dynamics on the barycentric coordinates in the case that $K \leq \dimD$.

\begin{theorem}[Delay embedding Theorem for Lipschitz maps~\cite{robinson}]
Let $A$ be a compact subset of a Hilbert space $H$ with upper box-counting dimension $d_{box}$, which has thickness exponent $\sigma$, and is an invariant set for a Lipschitz map $\textup{\gener}: H \rightarrow H$. Choose an integer $\memory > 2(d_{box}+\sigma)$, and suppose further that the set $A_\memory$ of $\memory$-periodic
points of $\textup{\gener}$ satisfies $d_{box}(A_\memory) < \memory/(2+\sigma)$. Then a prevalent set of Lipschitz maps $\textup{\observ}: H \rightarrow \R$ makes the $\memory$-fold observation map $\Phi: H \rightarrow \R^{\memory}$ defined in Eq.~\eqref{eq:delayembedding} injective on $A$.
\label{thm:robinsonTakens}
\end{theorem}

In finite dimensions, prevalent means almost every with respect to Lebesgue measure. The definition and more details on the thickness exponent and the box-counting dimension can be found in \cite{robinson} and Appendix~\ref{sec:App_Takens}. In essence, Theorem~\ref{thm:robinsonTakens} states that even observables that are not differentiable but Lipschitz continuous typically generate an embedding through the $\memory$-delay-coordinate map if $\memory$ is chosen to be sufficiently high. Further, Dellnitz et al.~\cite{dellnitz} observed that Theorem~\ref{thm:robinsonTakens} can be generalized to the case where the observable is multivariate, with an adjusted value for $\memory$\footnote[2]{For the required memory depth, let us observe that according to \cite{dellnitz}, the Theorem holds if we replace the observable $\observ:A\rightarrow \R$ by $L$ independent observables and construct a delay-map for this multivariate observable with $\memory > \frac{2(d_{box}+\sigma)}{L}$. Since for the $K$ barycentric coordinates one coordinate is always given by $1$ minus the sum of the others, they denote $K-1$ independent observables. We therefore deduce that a memory depth of $\memory > \frac{2(d_{box}+\sigma)}{K-1}$ should be enough. However, it is non-trivial to infer the thickness exponent $\sigma$ of the state space and therefore the necessary value for $\memory$.  Nevertheless, the Theorem merely gives an upper-bound for $\memory$ while one can hope that a smaller value makes for an injective delay-map. Therefore, often the most practical way to determine a well-fitting memory depth is simply by starting with an educated guess and iterating through increasing values until a sufficient quality is reached.}.

We want to apply Theorem~\ref{thm:robinsonTakens} to show that with memory we can recapture the dynamics on our data space in the path affiliation (Eq.~\eqref{eq:PSI_def}) based $K^M$-vertexed polytope. To this end, we denote
\begin{equation}
\mathcal{M}_{\Sigma}^\memory = \mathcal{M}_{\Sigma} \times \dots \times \mathcal{M}_{\Sigma}   
\end{equation}
as the \textit{product polytope} and further define
\begin{definition}
Let $\rho_{\Sigma}$ be the orthogonal projection to a polytope with vertices given by the columns of $\Sigma$. Let $A$ be a compact subset of a Hilbert space and let a function $\gener:A\rightarrow A$ be invertible. Then for $\memory \in \mathbb{N}$, the $\memory$-delay-coordinate map $\Phi^\memory_{\rho_\Sigma}:A\rightarrow \mathcal{M}_{\Sigma}^\memory$ as
\begin{equation}
   \Phi^\memory_{\rho_\Sigma}(\datapoint) = (\rho_{\Sigma}(\datapoint),\rho_{\Sigma}(\gener^{-1}(\datapoint)),\dots, \rho_{\Sigma}(\gener^{-\memory+1}(\datapoint)).
\end{equation}
\end{definition}
To extend Theorem~\ref{thm:robinsonTakens} for our aims, we will first prove that our projection, $\rho_{\Sigma}$, to the polytope is Lipschitz. Then show that its $M$-delay-coordinate map $\Phi^\memory_{\rho_\Sigma}$ on the product polytope $\mathcal{M}_{\Sigma}^\memory$ is injective. Then we extend the injection to the path affiliation polytope. This completes the proof showing the link between the full data space to the path affiliation polytope $\Psi(\mathcal{M}_{\Sigma}^\memory)$.

We begin by choosing our observable, with respect to Theorem~\ref{thm:robinsonTakens}, as the projection onto barycentric coordinates $\rho_\Sigma$ from Eq.~\eqref{eq:rhoK<D}, furthermore, it can be shown that this map is Lipschitz continuous.
\begin{lemma}
The function $\rho_\Sigma$ as defined in Eq.~\eqref{eq:rhoK<D} is Lipschitz continuous.
\label{lem:rholipschitz}
\end{lemma}
\begin{proof}
The polytope $\mathcal{M}_{\Sigma}$ is by construction closed and convex. The orthogonal projection onto such a set is always Lipschitz continuous with Lipschitz constant $1$.
\end{proof}

Using that the projection $\rho_\Sigma$ is Lipschitz continuous, we show that the $M$-delay-coordinate map is injective into the product space of the polytopes. 
\begin{theorem}
\label{thm:TakensPolytope}
Let $A$ and $\textup{\gener}$ be given as in Theorem~\ref{thm:robinsonTakens}. Let the projection $\rho_\Sigma$ from Eq.~\eqref{eq:rhoK<D} be inside the prevalent set of Lipschitz maps from Theorem~\ref{thm:robinsonTakens}. Then for a sufficient memory depth $\memory$, the $\memory$-delay-coordinate map of $\Phi^\memory_{\rho_\Sigma}$ is an injection from $A$ to $\mathcal{M}_{\Sigma}^\memory$.
\end{theorem}
\begin{proof}
$\rho_\Sigma$ is Lipschitz continuous by Lemma~\ref{lem:rholipschitz}. It maps from $A$ to $\mathcal{M}_{\Sigma}$, so that $\Phi^\memory$ maps from $A$ to $\bigotimes_{i=1}^\memory \mathcal{M}_{\Sigma}$. Then the claim is a direct consequence of Theorem~\ref{thm:robinsonTakens} and the assumption that $\rho_\Sigma$ is inside the prevalent set of Lipschitz maps.
\end{proof}

This theorem yields that sequences of delayed projections of the data in the polytope should be faithful representations of the corresponding states in the full data space. We can easily extend this representation into the path affiliation polytope.
\begin{corollary}
The function $\memop^\memory \circ \Phi^\memory$ is an injection from $A$ into the $K^\memory$-vertexed polytope of path affiliations.
\label{cor:psiphiinjective}
\end{corollary}
\begin{proof}
We map a point $\datapoint \in A$ to $\Phi^\memory(\datapoint)$ and apply $\memop^\memory$ to it. By definition of $\memop^\memory$, the result lies inside the path affiliation polytope. $\Phi^\memory$ is injective by Theorem~\ref{thm:TakensPolytope}. $\memop^\memory$ is injective by Proposition~\ref{pro:psiinjective}. The concatenation of injective functions is injective again so that $\memop^\memory \circ \Phi^\memory$ is injective on $A$.
\end{proof}

We have shown we can project a dynamical system on a compact manifold into a convex polytope, then use memory in the polytope and the path affiliation map to create an injection. Now we will show that using this injection we can construct equivalent dynamics of the true system in the path affiliation polytope.
\begin{theorem}
Let $A$ be given as in Theorem~\ref{thm:robinsonTakens} and let $\textup{\gener}:A\rightarrow A$ generate a discrete-time dynamical system on $A$, denoted by $\textup{\gener}(\datapoint_{t-1}) = \datapoint_t$. Let the projection $\rho_\Sigma$ onto barycentric coordinates of the \ref{eq:SPA1} polytope be inside the prevalent set of maps from Theorem~\ref{thm:robinsonTakens}. Then there exists an operator \textup{\textbf{v}} so that for a path affiliation vector $\psi^\memory_{t-1} = \memop^\memory(\gamma_{t-1},\dots,\gamma_{t-\memory})$, it holds $\textup{\textbf{v}}(\psi^\memory_{t-1}) = \rho_\Sigma(\textup{\gener}(\datapoint_{t-1})) = \gamma_t$.
\label{thm:SPATakens}
\end{theorem}
\begin{proof}
\begin{align}
     \gamma_t &= \rho_\Sigma(\datapoint_t)\\
     &=\rho_\Sigma(\gener(\datapoint_{t-1})),\\
      \intertext{using the injectivity of $\memop^\memory$}
     &=[\rho_\Sigma \circ \gener \circ(\Phi^\memory)^{-1}](\gamma_{t-1},\dots,\gamma_{t-\memory}),\\
     \intertext{using the injectivity of $\memop^\memory$ and therefore existence of its inverse on path affiliation vectors,}
     &= \underbrace{[\rho_\Sigma \circ \gener \circ(\Phi^\memory)^{-1}\circ (\memop^\memory)^{-1}]}_{=:\textbf{v}} (\psi^\memory_{t-1}).
     \label{eq:memorydynamics}
\end{align}
Note that by Corollary~\ref{cor:psiphiinjective} the last two steps could be merged into one. We have separated them to make the proof easier to follow and illustrate that points in $A$ can be injectively mapped to the product polytope and from there to the path affiliation polytope.
\end{proof}
The dynamical system Eq.~\eqref{eq:closedpsidynamics} arises through
\begin{equation}
\psi_t^\memory = [\Psi^\memory \circ \Phi^\memory \circ F\circ(\Phi^\memory)^{-1}\circ (\Psi^\memory)^{-1}] (\psi^\memory_{t-1}).
\end{equation}

There are two key points which have been derived in this subsection. Firstly, we can interpret the projection to the polytope derived in \ref{eq:SPA1} as an observable of the dynamical system. Hence, we can translate classical machinery like Takens' Theorem to this setting. Secondly, we can observe from Theorem~\ref{thm:SPATakens} that the mSPA propagator $\hat{\Lambda}^\memory$ is in essence a linear approximation of the operator~\textbf{v}.

\begin{remark}
Alternatively to approximating the function \textbf{v} with $\hat{\Lambda}^\memory$, one could approximate it using a neural network, whose capacity to reconstruct nonlinear dynamics has been well documented over the last years (see, e.g.,~\cite{vlachas}), including using delayed coordinates~\cite{marquez,xunbi}. We have employed this approach on two numerical examples, generating accurate forecasts of dynamics with lower memory depth compared to the one needed for the linear mapping (see Appendix~\ref{sec:App_NN}). At the same time, we have chosen a network structure that, as the column-stochastic linear mapping, keeps the dynamics inside the polytope. However, neural networks often require a highly cost-expensive search for optimal parameters and a good choice of their structure. Therefore, and for the reasons explained above, for now we focus on finding a linear mapping as an approximation for \textbf{v}. Constructing suitable neural networks to this end is a topic of future research.
\end{remark}

\FloatBarrier
\subsection{SPA in the context of Perron--Frobenius and Koopman operators}

Let us investigate the sufficiency of representing \emph{non-linear} dynamics by the proper choice of a \emph{linear} transformation $\hat{\Lambda}^\memory$ in~\ref{eq:mSPA2} from a different perspective than above. Note that this question is reminiscent of the well-known paradigm that so-called transfer operators (Perron--Frobenius and Koopman) are linear representations of arbitrary non-linear dynamics at the cost of acting in infinite dimensions. Indeed, the Perron--Frobenius operator $\mathcal{P}$, that describes the action of the dynamics $\gener$ on distributions, is defined on Dirac measures via~$\mathcal{P} \delta_x = \delta_{\gener(x)}$ for all~$x$.

In which sense do the affiliations $\gamma^i_t$ tell us something about the distribution of the system? For this, let us assume for simplicity that the SPA polytope $\mathcal{M}_{\Sigma}$ with vertices $\sigma_j$ is a full-dimensional simplex and denote by $\varphi_i: \mathcal{M}_{\Sigma} \to [0,1]$ the unique linear functions with $\varphi_i(\sigma_j) = \delta_{ij}$ (the usual ``hat functions''). Denoting the $i$-th barycentric coordinate of the point $x \in \mathcal{M}_{\Sigma}$ by $\gamma^i(x)$, we have $\gamma^i(x) = \varphi_i(x) = \int_{\mathcal{M}_{\Sigma}} \varphi_i(y) \delta_x(y)\,dy$.
If a trajectory of the system starts in a point $x \in \mathcal{M}_{\Sigma}$ and stays there for all times $t\ge 0$, the approximation in mSPA is built by the observables~$\gamma_t^i = \int_{\mathcal{M}_{\Sigma}} \varphi_i(y)\, [\mathcal{P}^t\delta_x](y)\,dy$.

It is thus natural to ask whether the observations $\gamma^i_t$ are sufficient to represent~$\mathcal{P}$. To this end, it is necessary that $(\gamma_t)_{t\in\mathbb{N}}$ can represent the natural space on which~$\mathcal{P}$ acts; this can for measure-preserving systems be taken to be~$L^2(\mathcal{M}_{\Sigma})$, cf.~\cite{KlKoSch16}. Hence, is the mapping $\smash{ L^2(\mathcal{M}_{\Sigma}) \ni f \mapsto 
\big( \int_{\mathcal{M}_{\Sigma}} \varphi_i\, \mathcal{P}^tf \big)_{t\in \mathbb{N}} }$ injective? Note, that for simplicity we can consider this question for a fixed~$i$. By the duality of $\mathcal{P}$ and the Koopman operator~$\mathcal{K}$. i.e.\ $ \int_{\mathcal{M}_{\Sigma}} \varphi_i\, \mathcal{P}^tf = \int_{\mathcal{M}_{\Sigma}} \mathcal{K}^t\varphi_i\, f$, the equivalent question is: Is $\mathrm{span} \{ \varphi_i, \mathcal{K}\varphi_i, \mathcal{K}^2\varphi_i,\ldots \}$ a dense subspace of~$L^2(\mathcal{M}_{\Sigma})$?
This latter question has been investigated in ergodic theory~\cite[Chapter~18]{eisner2015operator}, but the technical nature of the results is beyond the scope of the present work.

In conclusion, it seems reasonable to restrict our attention in Eq.~\eqref{eq:mSPAsystem} to linear transformations $\hat{\Lambda}^\memory$ to find a sufficiently powerful representation of the system at hand. The assumptions under which this is possible will be investigated elsewhere.
In Appendix~\ref{sec:App_Koopman} we discuss the connection between mSPA and methods which focus on finding a finite-dimensional approximation of the Koopman operator.

In the following section, we will put together all the SPA methods into a single numerical scheme.

\section{Learning and Lifting}
\label{sec:scheme}
We have now introduced mSPA as a method to reconstruct dynamics in the barycentric coordinates. We will now place it within a numerical scheme which reconstructs dynamics in the full space. In particular, we will use \ref{eq:SPA1} to project to a low-dimensional polytope and mSPA to learn the dynamics. We will introduce an additional mSPA-like step to lift the predicted barycentric coordinates back to the full space.

\subsection{Combining mSPA and SPA~I}
In the mSPA construction we build a map into a small polytope, $K \ll D,$  from the $K^M$ path affiliation polytope. We showed in Eq.~\eqref{eq:closedpsidynamics} that we could construct a dynamical system on the $K^M$ path affiliation polytope. Now we want to extend this map to give corresponding predictions in the original data space. We do this by introducing a new polytope of dimension $K',$ which we call the \emph{lifting polytope}, such that the projection loss between the data and its representation in the lifting polytope is minimal. Then using SPA II we can learn a column-stochastic matrix which maps from the memory polytope to the lifting polytope (see Figure~\ref{fig:mSpa_diagram_reduced}).
This problem reads,
\begin{equation}
\begin{split}
\hat{\Lambda}'^\memory &= \argmin_{\Lambda^*} \Vert [\gamma'_{\memory} \vert \cdots \vert\gamma'_\timeT] - \Lambda^* [\psi_\memory \vert \cdots\vert\psi_{\timeT}] \Vert_F\\
\text{satisfying,} & \\
& \hat{\Lambda}'^\memory \geq 0 \text{ and } \sum_{k=1}^{K} \hat{\Lambda}'^\memory_{k,\bullet} = 1. 
\end{split}
\label{eq:spa2lifting}
\tag{SPA Lifting}
\end{equation}
where $\gamma'$ denotes the barycentric coordinates in the lifting polytope with vertices $\Sigma'$ and subject to the typical constraints on $\Lambda^*.$ 

Naturally, there are many possible ways of mapping from the memory polytope to the original data space. The simplest is arguably given by solving the problem
\begin{equation}
\argmin_{A^*} \Vert [\datapoint_{\memory} \vert \cdots \vert\datapoint_\timeT] - A^* [\psi_{\memory} \vert \cdots\vert\psi_{\timeT}] \Vert_F
\label{eq:linearmapping}
\end{equation}
to construct a linear propagator from path affiliations to subsequent barycentric coordinates (this is simply the \eqref{eq:spa2lifting} problem without the constraints on the solution). However, we found that if the original data has low rank, this optimization problem is badly-conditioned.
It turned out in many examples that this caused the ensuing best-fitting matrix to give highly unstable mappings since it is prone to contain very high coefficients. This did not happen when mapping into the lifting polytope first. In this way, the intermediate step of mapping into the lifting polytope can be seen as a preconditioning step.
We present the algorithmic steps of the methods in Algorithm~\ref{alg:mSPAscheme}. The colors used in the algorithm correspond to the ones from Figure~\ref{fig:mSpa_diagram_reduced} from the Introduction.
\begin{algorithm}[]
\SetKwInOut{KwInput}{Input}
\KwInput{data points $\datapoint_1,\dots,\datapoint_T \in \R^{\dimD}$, $K,K',\memory \in \mathbb{N}$}

\SetKwInOut{KwInput}{\textit{Learning of operators}}
\KwInput{}
\vbox{\colorbox{YellowGreen!20}{Solve \eqref{eq:SPA1} for $K$ to obtain $[\Sigma,\Gamma]$ \quad(\text{Learning polytope})}}
\vbox{\colorbox{YellowGreen!20}{Solve \eqref{eq:SPA1} for $K'$ to obtain $[\Sigma',\Gamma']$ \quad(\text{Lifting polytope})}}
\vbox{\colorbox{Dandelion!20}{Solve \eqref{eq:mSPA2} for $\Gamma$ to obtain $\hat{\Lambda}^\memory$ \quad(Learn propagator on learning polytope)}}
\vbox{\colorbox{Thistle!20}{Solve \eqref{eq:spa2lifting} to obtain $\hat{\Lambda}'^\memory$ (Learn map from learning to lifting polytope)}}
\SetKwInOut{KwInput}{\textit{Prediction}}
\KwInput{}
For starting values $\gamma_1,\dots,\gamma_{\memory}$ on learning polytope compute $\psi^\memory_{\memory}$ as in Eq.~\eqref{eq:PSI_def} \\
 \For{$t = \memory+1:T_{end}$}{
  \vbox{\colorbox{Dandelion!20}{$\gamma_{t} \leftarrow \hat{\Lambda}^\memory \psi^\memory_{t-1}$ \quad\text{(Propagation on learning polytope as in Eq.~\eqref{eq:mSPAsystem})}}}
  
    \vbox{\colorbox{Dandelion!20}{ $\psi^\memory_t \leftarrow \memop^\memory(\gamma_t,\dots,\gamma_{t-\memory+1}) \quad$\text{(Computation of path affiliation)}}}
    
    \vbox{\colorbox{Thistle!20}{$\gamma'_t \leftarrow \hat{\Lambda}'^\memory \psi^\memory_t$ \quad\text{(Mapping to lifting polytope)}}}
  
  \vbox{\colorbox{YellowGreen!20}{$\datapoint_{t} \leftarrow \Smatrix' \gamma'_{t}$ \quad\text{(Lifting to original space)}}}
 }
 \caption{Numerical scheme}
 \label{alg:mSPAscheme}
\end{algorithm}

\subsection{Examples}
We will now apply the Learning and Lifting approach with mSPA to two examples, the Kuramoto-Sivashinsky equation and the Chua circuit.
\begin{example}
\label{exa:KSlowdim}
\textbf{The Kuramoto-Sivashinsky equation }
Let us now consider the application of this approach on a nonlinear variant of the Kuramoto--Sivashinsky (KS) equation~\cite{sivashinsky,kuramoto}. We  first solve the mSPA problem on our learning polytope, choosing a memory depth required to reconstruct the nonlinear dynamics. Then we will build our lifting polytope such that there is minimal projection loss, and finally, we will learn the map between the two polytopes. We will compare our approach with other well-known methods.

The Kuramoto--Sivashinsky (KS) equation is a one-dimensional fourth-order PDE and has become a typical example of complex dynamics in recent research on system identification methods (see \cite{lulin2015,linfu}; the relation between mSPA and their method will be discussed in Appendix~\ref{sec:App_Koopman}). It reads
\begin{equation}
    u_t + 4u_{xxxx} + 16 u_{xx} + 8(u^2)_x = 0,\quad t \geq 0,
\end{equation}
for $x \in [0,2\pi)$ with periodic boundary conditions. We discretized the interval $[0,2\pi)$ into 100 equidistant grid points and perform numerical integration with a time step of $0.001$. As initial values we chose $u(x) = 0.0001\cos(x)(1+\sin(x))$. In order to approximate the solution of the PDE we used a 4th-order exponential time differencing Runge--Kutta (ETDRK4)~\cite{coxETDRK4} on the $100$ grid points. Simulating the dynamics until a time of $6.5$ and truncating states from before $2.5$, we obtained $4000$ time steps. With this, we obtained a $100$-dimensional system of which each coordinate represented one grid point.

The PDE solution showed the behaviour of a travelling wave: high values move from top to bottom in the interval and are passed from the lower boundary to the upper boundary (see Figure~\ref{fig:KS_gammaprediction_K2_tau001}~\textbf{a.}). We projected the data onto polytopes with increasing number of vertices $(K)$ and saw that the projection loss decreased with increase in $K$ with an elbow point around $K=7$ (see Figure~\ref{fig:KS_gammaprediction_K2_tau001}~\textbf{b.}). Then we investigated the memory depth required for the worst case polytope ($K=2$) to recapture the periodic behaviour (see Figure~\ref{fig:KS_gammaprediction_K2_tau001}~\textbf{c.}). We saw that we needed memory depth of $\memory=10$ to be able to construct a propagator which could reconstruct the periodic behaviour in the data (see Figure~\ref{fig:KS_gammaprediction_K2_tau001}~\textbf{d.}). 

\begin{figure}[ht]
\centering
\includegraphics[width = 0.9\textwidth]{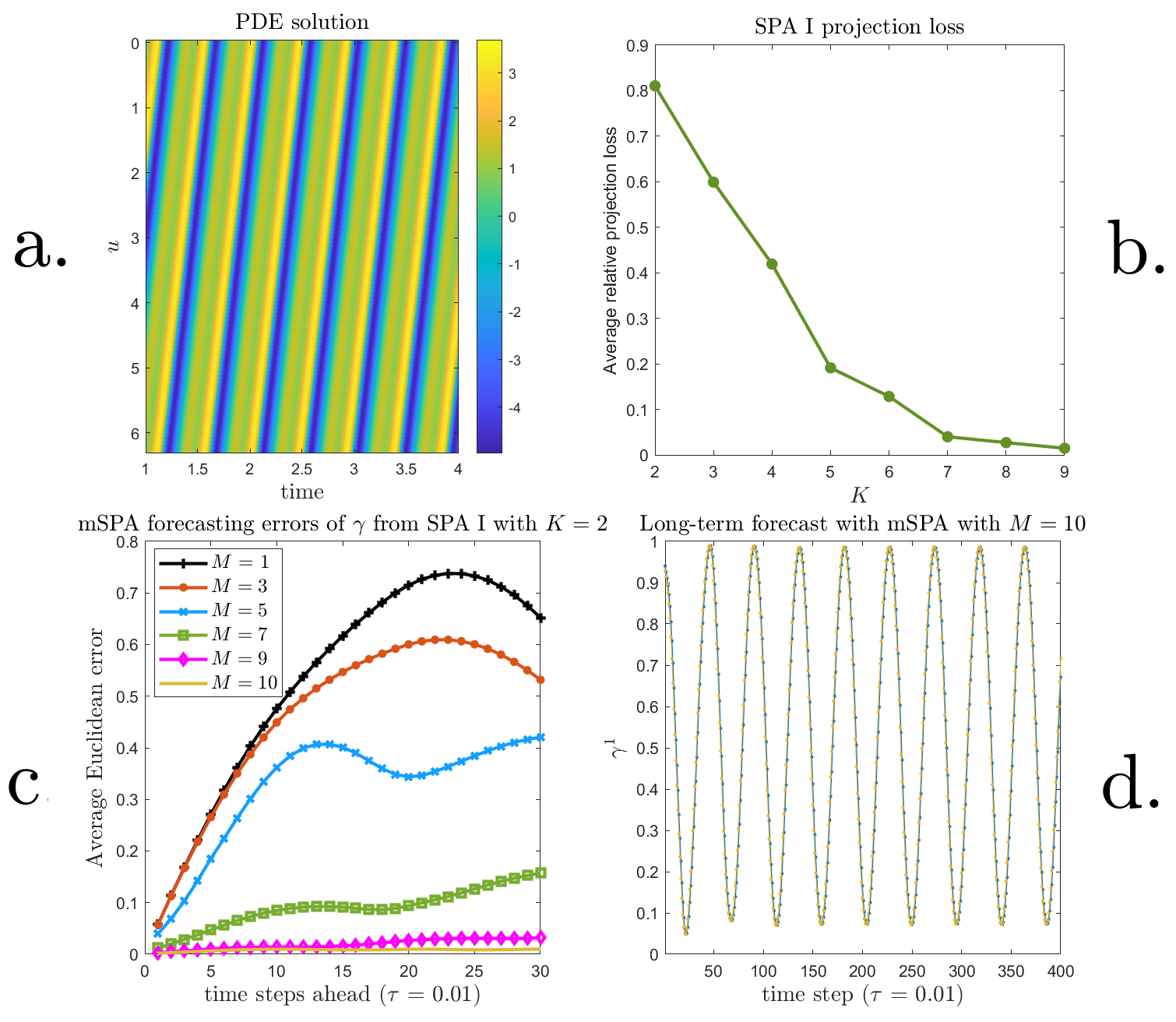}
\caption{\textbf{a.}: evolution of the KS PDE for 4000 time steps of length $0.001$. \textbf{b.}: projection loss of the data with different number of vertices. \textbf{c.}: average \tgammaerror with mSPA with different memory depths $\memory$ with time step of length $\tau = 0.01$. The 100-dimensional data were projected onto a line, i.e., we used $K = 2$. For high memory depths, the errors stayed small for increasing forecast width. For low memory depths, the prediction quickly converged to a fixed point (as in Figure~\ref{fig:Lorenz96_x1x4_Fullplot} from Example~\ref{exa:overlap}). As a consequence the error is the distance between a wave-like and a constant function, making for a wave-like error function. \textbf{d.}: $\gamma^1$-coordinate of long-term forecast of $\gamma$ using mSPA with $\memory = 10$. Prediction in yellow and true trajectory in blue. We can see that the amplitude and frequency were approximately met in the barycentric coordinates.}
\label{fig:KS_gammaprediction_K2_tau001}
\end{figure}
\end{example}

We can see in the above example, that even after projecting the full state space onto a line, we could use memory to reconstruct the nonlinear behaviour in the data. However, because of the projection onto the line, we cannot interpret or predict the trajectories of the propagator with respect to the real data space. Given we are only predicting on the learning polytope, we cannot compare the method to other methods. Hence, now to have interpretability and map the trajectories back into the real data space, we create our lifting polytope. We chose to have 8 vertices, given it had a projection loss of only $\approx 2\%$  (Figure~\ref{fig:KS_gammaprediction_K2_tau001} \textbf{b.}). 

Now that we have a way to map back into the data space, we will now compare the accuracy of the mSPA method against the known linear autoregressive (AR)~\cite{brockwell} models with varying memory depths and Sparse Identification of Nonlinear Dynamics (SINDy)~\cite{sindy} with nonlinear basis functions. Their model formulations are given Appendix~\ref{sec:App_Koopman}.

\begin{example}
\label{exa:KSlifting}
\textbf{Reconstructing the dynamics of the KS PDE }
We first investigated the linear AR method. In Example~\ref{exa:KSlowdim} we observed that memory of $10$ was needed to capture the nonlinearity in the data with mSPA if $K=2$. Hence, we first trained a linear AR on the real data with a memory of $M=10$. We saw that the memory size was insufficient to give an accurate reconstruction of the periodicity (see Figure~\ref{fig:KS_allmethods_Comparison}~\textbf{a.} I and II). When we studied the long-term behaviour, that is, we trained up to time $T=3$ and forecast up to time $T=20,$ we saw that the prediction deviates from the true solution and converges to a point (see Figure~\ref{fig:KS_allmethods_Comparison}~\textbf{a.}, III and IV). To find the necessary memory depth needed for the linear AR method, we increased the memory incrementally and found that a memory depth of $M=15$ gave a very accurate approximation (see Figure~\ref{fig:KS_allmethods_Comparison}~\textbf{b.}, I and II). Furthermore, the long-term forecasting of the method was stable and we could see the Euclidean error between the real trajectory and the predicted trajectory plateaued (see Figure~\ref{fig:KS_allmethods_Comparison}~\textbf{a.}, III and IV).

We saw that the nonlinearity of the problem demanded a high memory depth for the linear AR method, hence, we investigated which nonlinear functions would be necessary to capture the dynamics without memory using SINDy. We tested different basis functions, starting from linear to polynomials of degree two and three. The approximations diverged to infinity and also adding higher degree polynomials made the problem quickly intractable (results not shown). We could get a very good approximation when we considered specifically only quadratic polynomial basis combined with very particular trigonometric functions (see Appendix~\ref{sec:App_techsKS}). When we trained on these basis functions, we found that the data was very well approximated (see Figure~\ref{fig:KS_allmethods_Comparison}~\textbf{c.}, I and II), and furthermore, the long term forecasting was also very stable and accurate (see Figure~\ref{fig:KS_allmethods_Comparison}~\textbf{a.} III and IV). That said, SINDy is capable to reconstruct the dynamics very precisely but for this example requires a \textit{very} specific set of basis functions. We can see that if we build a linear approximation (like the AR method), we need sufficient memory to capture the periodicity. We can circumvent memory (e.g., by using SINDy), but then one needs to find very specific and highly nonlinear basis functions to be able to capture the periodicity.

Now considering the same problem with mSPA with a lifting polytope, we found that we could capture the periodicity with a smaller memory depth. We fixed our lifting polytope to have eight vertices ($K' = 8$). To verify that the transformation to barycentric coordinates in the lifting polytope was not by itself supporting the reconstruction of dynamics, we trained a linear AR model with different memory depths on the lifting polytope and found that, even with a memory depth of 20, we did not see stable periodic predictions (see Figure~\ref{fig:KS_allmethods_Comparison}~\textbf{d.}, I--IV). Furthermore, we also used SINDy on the lifting polytope with up to quadratic basis functions and found that the approximations diverged (results not shown). We then tested two scenarios: the first a two-vertexed learning polytope with memory depth of 10 and the second a three-vertexed learning polytope with a memory depth of 6. In the case of $(K=2$, $M=10)$, by using the lifting polytope we could extrapolate our trajectories from the memory polytope onto the real data space. We saw that we could recapture the periodic behaviour and amplitude (see Figure~\ref{fig:KS_allmethods_Comparison}~\textbf{e.}, I and II), however we saw a small shift in delay in the frequency as the system was forecast for longer (see Figure~\ref{fig:KS_allmethods_Comparison}~\textbf{e.}, III and IV). Then we trained using the case $(K=3$, $M=6)$ and found that the approximation was better than the $(K=2$, $M=10)$ case (see Figure~\ref{fig:KS_allmethods_Comparison}~\textbf{f.}, I and II), that is, we did not observe as strong a frequency shift in the approximation (see Figure~\ref{fig:KS_allmethods_Comparison}~\textbf{f.}, III and IV). The forecasting error in \textbf{e.} is larger than in \textbf{d.} but that seems to come from the shift in the frequency while the amplitude is well maintained which is not the case in~\textbf{d.} The same can be observed for~\textbf{f.} For comparison, the dimensions of the coefficient matrices for all methods used in this example are given in the caption of Figure~\ref{fig:KS_allmethods_Comparison}. Their total numbers lie in the same orders of magnitude except for SINDy which needs fewer coefficients but for which, as mentioned above, most sets of basis functions yielded diverging models.


\begin{figure}
\includegraphics[width = \textwidth]{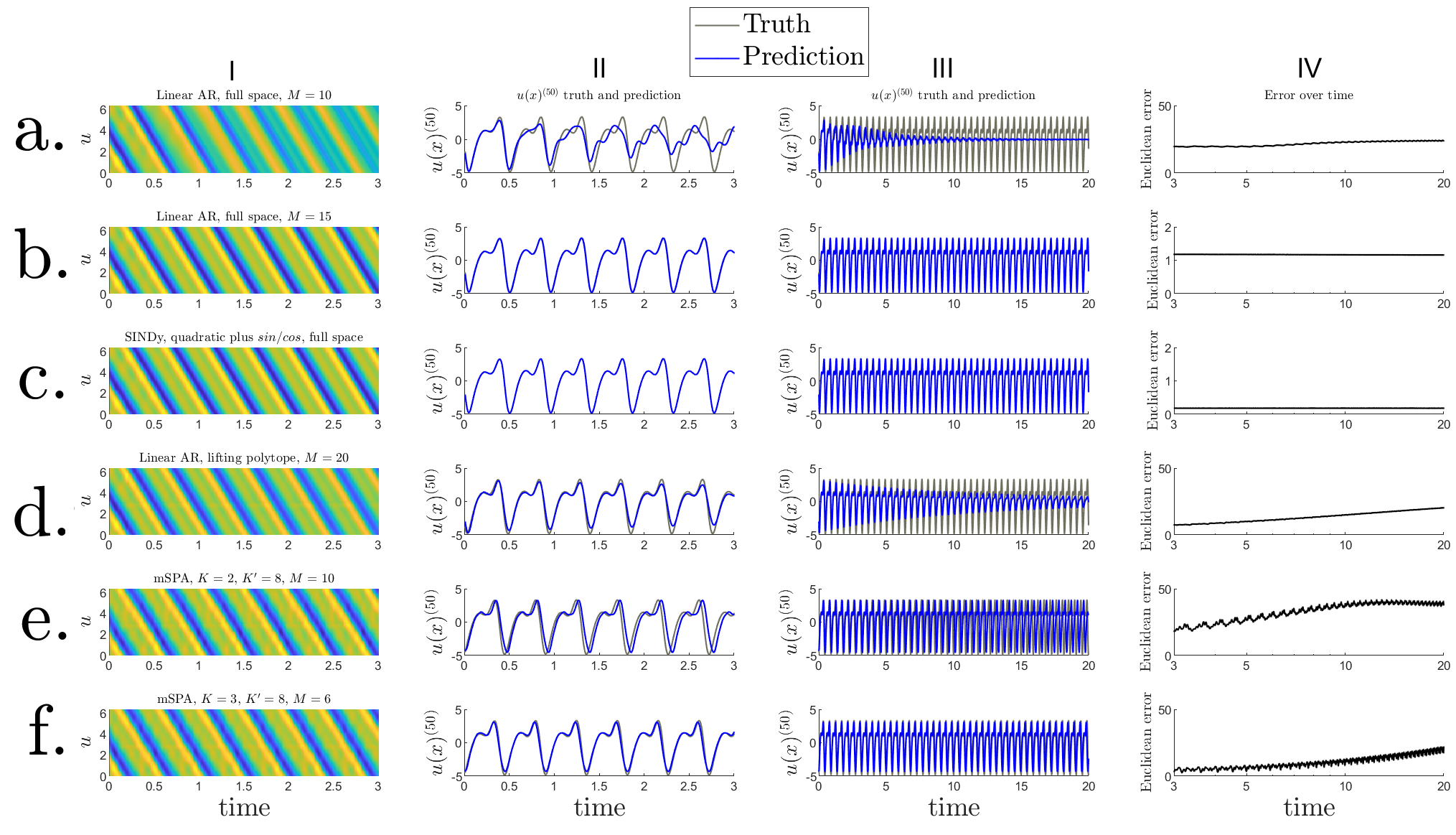}
\caption{\textbf{a.} A linear AR model on the full space with memory $\memory = 10$ generates quick convergence to a fixed point. \textbf{b.} A linear AR model with $\memory = 15$ well recreates the data in the long term with only a low error. \textbf{c.} The SINDy model with a very particular choice of basis functions (see Appendix~\ref{sec:App_techsKS}) generates a very accurate forecast. All models we found using other sets of basis functions diverged, however. \textbf{d.} On the lifting polytope, a linear AR model even with $\memory = 20$ cannot reconstruct the periodicity of the dynamics. \textbf{e.} With mSPA we are able to reconstruct the amplitude and approximately the frequency of the dynamics after \ref{eq:SPA1} projection onto only a line ($K=2$) and $\memory = 10$. \textbf{f.} With $K=3$ we only need 6 memory terms to gain a better approximation. The error increases towards the end of the frame because of a slight deviation in the frequency.\\
\textit{Number of coefficients needed:} \textbf{a.} ($100\times 1000)$-matrix, \textbf{b.} ($100\times 1500)$-matrix,  \textbf{c.} ($100\times 7050)$-matrix, \textbf{d.} $(8\times 160)$-matrix, \textbf{e.} ($2\times 1024)$-matrix and ($8\times 1024)$-matrix, \textbf{f.} ($3\times 729)$-matrix and ($8\times 729)$-matrix.}
\label{fig:KS_allmethods_Comparison}
\end{figure}
\end{example}

Collating the results, we can see that mSPA with the lifting polytope was able to reconstruct amplitude and periodicity of a nonlinear dynamical system. We were able to do it with less memory than the linear AR model. Furthermore, we were able to reconstruct the dynamics without prior knowledge on the dynamics, unlike SINDy, where we had to scan through many combination of basis functions. We only have three free parameters, the number of vertices for the learning manifold, the memory depth, and the number of vertices for the lifting polytope $(K,M,K').$ With these three, we could reconstruct nonlinear dynamics from data.   

As a last example, we consider a chaotic Lorenz-96 system and evaluate mSPA's capacity for short-term forecasts for complex systems.

\begin{example}
\label{exa:overlap}
\textbf{Lorenz-96, $\dimD = 10$, $F = 5$ (chaotic regime)}
We consider the Lorenz-96 system from Example~\ref{exa:overlap} but this time choose $\dimD = 10$ and $F=5$ which is sufficient to create chaotic behaviour as we can observe by verifying that the maximal Lyapunov exponent~\cite{cencini} of the ensuing trajectory is positive.

We create realisations with a time step of~$0.05$. We use one realisation of length $100$ (2000 time steps) for training of the mSPA model and create 50 realisations of length $7.5$ (150 time steps) each from randomly chosen starting points close to the Lorenz-96 attractor for testing. In \ref{eq:SPA1}, we chose $K = 3$ and $K' = 8$ to determine learning and lifting polytopes. The latter produces a projection error of 23\%. In mSPA, we use the memory depth~$\memory = 7$. 
As in the previous examples, we have normalized the data into the unit cube $[-1,1]^\dimD$ by a linear transformation.

We compare the autocorrelations of the testing trajectories and their predictions. For a set of $N_{\text{test}}$ trajectories of length $T$ time steps each in the learning polytope, we define the autocorrelation of the $i$th coordinate for a time lag of $l$ time steps as
\begin{equation}
    a_i^{l} := \frac{1}{N_{\text{test}}(T-l)}\sum_{r=1}^{N_{\text{test}}} \sum_{t=1}^{T-l} ((\gamma_r)_t^i-\bar{\gamma^i}) ((\gamma_r)_{t-l}^i-\bar{\gamma^i}).
\end{equation}
The term $(\gamma_r)_t$ is the $t$th point in the $r$th testing trajectory and $\bar{\gamma}$ is the mean of $\gamma$ across all testing trajectories. For $\gamma'$ (lifting polytope) and $\datapoint$ (full space), the definition is analogous. The results are shown in Figure~\ref{fig:L96_autocorr}.

We can see that, similarly to the Kuramoto--Sivashinsky PDE, mSPA produces accurate predictions in the learning polytope (Figure~\ref{fig:L96_autocorr},~\textbf{a.}).  In the lifting polytope and the full space, an additional error is introduced by the non-trivial lifting step. The autocorrelations are well met (Figure~\ref{fig:L96_autocorr},~\textbf{b.}--\textbf{d.}) in learning and lifting polytopes and the full space, except for short time lags for lifting polytope and full space. This example emphasizes that mSPA, while not precisely recreating individual realizations of chaotic dynamics, is capable of reproducing quantities related to the statistical behavior.
\begin{figure}[ht]
\includegraphics[width = \textwidth]{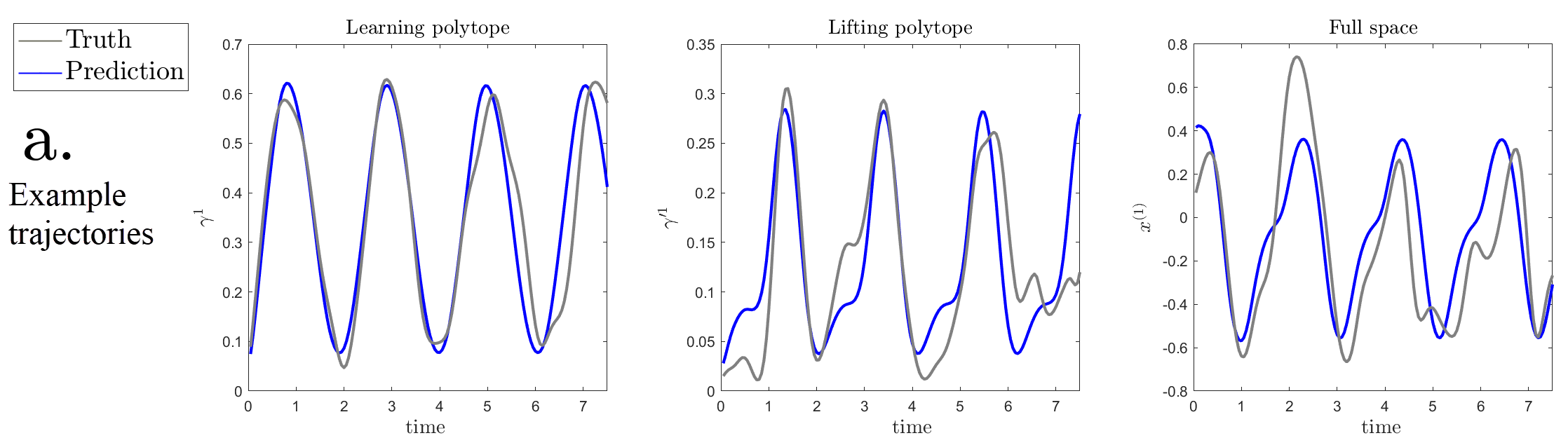}
\includegraphics[width = \textwidth]{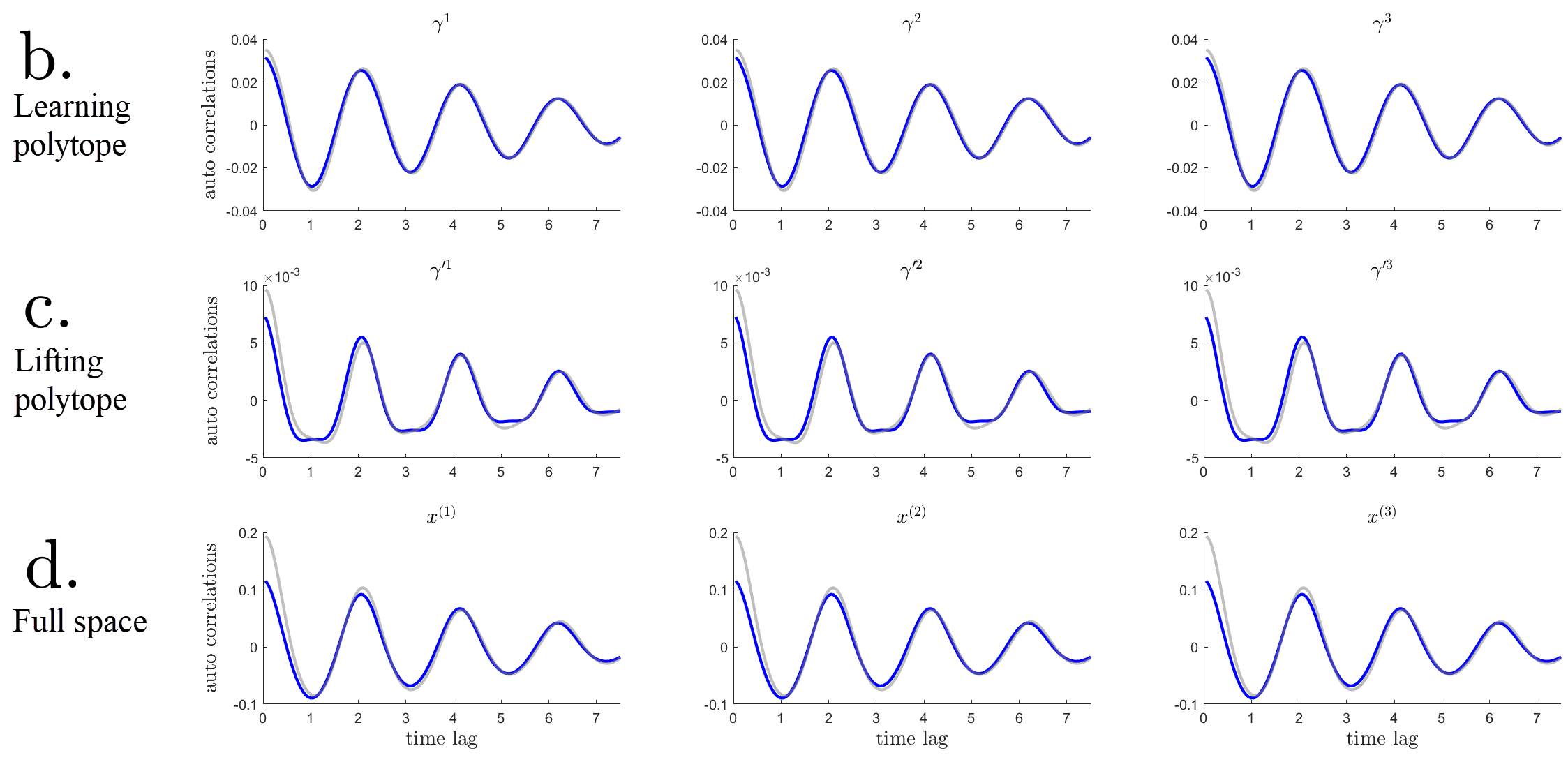}
\caption{Results of mSPA on the chaotic Lorenz-96 variant. 
\textbf{a.} Examples of forecasts. Prediction in blue, true data in grey. \textbf{b.}--\textbf{d.} Autocorrelations in learning and lifting polytopes and full space. For the sake of visualization, only the first three coordinates of each are shown. Note that the autocorrelations are not normalized by the variances of the trajectories as in the Pearson correlation coefficient, so that the orders of magnitude on the $y$-axes offers no immediate interpretation.}
\label{fig:L96_autocorr}
\end{figure}
\end{example}

In Appendix~\ref{sec:App_chua}, we show an additional example given by the Chua circuit, a three-dimensional chaotic attractor, to demonstrate the capacity of mSPA to reconstruct dynamical behavior on an attractor having multiple connected components. There we also demonstrate an additional modelling step by distinguishing between a small \textit{forward time step} and a larger \textit{memory time step}, the latter denoting time separation between points in the delay-coordinate map.

We conclude this section with some theoretical and numerical observations on optimal choices for the parameters of \ref{eq:SPA1} and \ref{eq:mSPA2}.

\FloatBarrier

\subsection{Theoretical observations on the choice of parameters}

We derived the evolution of $\gamma$ in the form of $\gamma_t = \textbf{v}(\gamma_{t-1}\,\dots,\gamma_{t-\memory})$ (see Eq.~\eqref{eq:memorydynamics}). We can find for the approximation error of mSPA
\begin{equation*}
    \begin{split}
        \gamma_t - \hat{\Lambda}^\memory \psi^\memory(\gamma_{t-1},\dots,\gamma_{t-\memory}) &= \gamma_t - \sum_{i=1}^{K^\memory} \Lambda_{| i} (\psi^\memory_t)_i\\
        &= \sum_{i=1}^{K^\memory}(\textbf{v}(\gamma_{t-1}\,\dots,\gamma_{t-\memory}) - \Lambda_{| i} )  (\psi^\memory_t)_i
    \end{split}
\end{equation*}
where $\Lambda_{| i}$ denotes the $i$th column of $\hat{\Lambda}^{\memory}$.

We can observe from this that we require $\Lambda_{| i}$ to be close to $\textbf{v}(\gamma_{t-1}\,\dots,\gamma_{t-\memory})$ whenever $(\psi^\memory_t)_i$ is high, especially when all other entries of $(\psi^\memory_t)$ are low. As a consequence, it is desirable if the entries of $\gamma$ differ from each other sufficiently throughout the trajectory because otherwise all $(\psi^\memory_t)_i$ are comparably high at all times and the predictions made by multiplication with $\hat{\Lambda}^{\memory}$ are similar regardless of the actual value of the dynamics. This means that the vertices from the solution of \ref{eq:SPA1} should not be too far apart from the domain of the dynamics but rather chosen as tight as possible. Further, the dynamics need to be sufficiently smooth for mSPA to generate accurate predictions, which clearly is the case for the majority of dynamical modelling methods. Suppose that two sequences $[\datapoint_{t-1},\dots,\datapoint_{t-\memory}],[\datapoint_{\tilde{t}-1},\dots,\datapoint_{\tilde{t}-\memory}]$ are close to each other so that also $\psi^\memory_{t-1}$ and $\psi^\memory_{\tilde{t-1}}$ are similar to each other. Suppose the subsequent values $\datapoint_t$ and $\datapoint_{\tilde{t}}$ are far apart, as possible in, e.g., chaotic dynamics. Let us assume that only the $i$th entry of the path affiliations is considerably high (this is a rare case but serves an illustrative purpose here; it can, however, happen, if the points in the two sequences are close to a vertex). The mSPA prediction then cannot cover both cases -- the behaviour of the dynamics at times $t$ and $\tilde{t}$ -- since it is approximately given by the constant vector $\Lambda_{|i}$. Of course, generally multiple entries of $\psi^\memory$ will be of high influence so that $\hat{\Lambda}^\memory$ can be chosen in the best possible way to balance the different cases. However, it seems advisable if the vertices are not placed close to bifurcation points of the dynamics because there the scenario explained above can happen: barycentric coordinates and therefore path affiliations can be similar while subsequent values of the dynamics are very different from each other.\\

Although one could perceive the mSPA formulation of dynamics as a mere coordinate change, it proved challenging to find an analytical representation even of linear systems in the mSPA form of $\gamma_t = \hat{\Lambda}^\memory \psi^\memory(\gamma_{t-1},\dots,\gamma_{t-\memory})$ with or without memory. It seems that mSPA is \textit{not} a method that exactly represents simple systems through coordinate changes. Rather, it is an \textit{approximative} method -- with a very specific model form -- to forecast systems of varying complexity, as the examples have shown.

Together with our findings on the Lorenz-96 system, the KS equation, the Chua circuit and investigations on simple harmonic oscillators (not shown in the paper), we derived the following intuitions to generally hold true. To obtain high quality using mSPA:
\begin{itemize}
    \item Choose the polytope from the \ref{eq:SPA1} problem and scale it such that the trajectory of the barycentric coordinates is as close as possible to the boundary of its domain, the $K$-vertexed unit simplex, without touching it.
    
    \item Choose the memory time step (see Appendix~\ref{sec:App_chua}) so that the image of the delay-coordinate map using this memory lag has similar geometry as the original attractor. Further, it seemed that typically, if each sequence of length $\memory$ in the full space data has at least one point that is close to or inside the polytope, mSPA worked well.
    
    \item Choose the memory depth as high as needed to obtain sufficient quality of the model. Clearly, this is constrained by the exponential growth of the term~$K^\memory$.
\end{itemize}

\FloatBarrier

\section*{Conclusion}

We have introduced a new data-driven method to model nonlinear dynamical systems. For this method, we project data to a lower-dimensional polytope using the novel SPA method. We then model the dynamics on the reduced barycentric coordinates in the polytope, thereby creating a stable dynamical system. To reconstruct the information of the data that is lost by the projection, we use memory of the barycentric coordinates. We defined the path affiliation function which measures the closeness of a sequence of points to each possible sequence of vertices of the polytope and construct a linear column-stochastic matrix that pushes the dynamics forward in time. Therefore, the method does not require prior intuition of the dominant structures behind the dynamics but instead uses a fixed observable.

We showed that (under standard assumptions) there exists a dynamical system on the path affiliations that is topologically conjugate to the original one using a variant of Takens' Theorem. We further showed in several examples that our method can accurately reconstruct nonlinear dynamical systems of varying dimension and complexity.

In the future, several issues should be addressed: firstly, as pointed out before, Takens gives a potentially nonlinear function between path affiliations and subsequent barycentric coordinates which we approximate by a linear mapping. As briefly discussed earlier, one could replace the linear mSPA propagator by a suitable nonlinear function, e.g., as done in this article using a neural network whose output always gives stochastic vectors. The examples in this article have demonstrated that this could be a worthwhile endeavor. Further, the computation of the linear mSPA propagator could need simplification since the dimension of the path affiliation vectors grows exponentially with memory depth.
Moreover, we observed that the choice of the polytope could strongly influence the quality of the ensuing mSPA reconstruction. The \ref{eq:SPA1} problem produces a polytope which has minimal total distance to the data. It is, however, unclear at this point, if a polytope can be chosen so that the estimation of dynamics becomes easier with mSPA, except for our observation at the end of the previous section. Potentially, the \ref{eq:SPA1} problem  could be modified, e.g., regularized, in a way to produce a polytope that is optimal also to capture the dynamics. Additionally, mSPA could be extended for dynamical systems with control input. For this aim, one could solve the \ref{eq:SPA1} problem for the control input space, too, translate inputs to barycentric coordinates and formulate an augmented mSPA problem.

\section*{Acknowledgements}

We thank Illia Horenko from USI Lugano for helpful discussions regarding SPA and its variants and Raphael Gerlach from University of Paderborn for the helpful discussions on the numerical solution of the Kuramoto--Sivashinsky equation and provision of the Matlab code.
PK, NW and CS have been partially supported by Deutsche Forschungsgemeinschaft (DFG) through grant CRC 1114 ``Scaling Cascades in Complex Systems'', Project Number 235221301, Projects A01 and B03, and under Germany’s Excellence Strategy -- The Berlin Mathematics Research Center MATH+ (EXC-2046/1 project ID:  390685689). VS was funded by the Berlin Institute for the Foundations of Learning and Data (BIFOLD) (Project BIFOLD-BZML 01IS18037H).
\appendix
\section{Properties of SPA and memory SPA}

\subsection{Technical details of $\rho_\Sigma$ for the case if $K \leq \dimD$}
\label{sec:App_rhosigma}
Note that in the case $K \leq D$, $\mathcal{M}_{\Sigma}$ can be assumed to be a simplex.

\begin{lemma}
Let $\Sigma = [\sigma_1,\dots,\sigma_D] \in \R^{\dimD \times K}$ denote the vertices of a non-simplex polytope $\mathcal{M}_{\Sigma}$. Then there exists a simplex with vertices $\tilde{\Sigma} \neq \Sigma$ so that for points $[\datapoint_1\vert\cdots\vert\datapoint_T] \in \R^{\dimD}$, there exist barycentric coordinates $[\gamma_1\vert\cdots\vert\gamma_T] \in \R^K$ so that $\Vert [\datapoint_{1} \vert \cdots \vert\datapoint_\timeT] - \tilde{\Sigma} [\gamma_1\vert\cdots\vert\gamma_T] \Vert_F \leq \Vert [\datapoint_{1} \vert \cdots \vert\datapoint_\timeT] - \Sigma [\gamma_1\vert\cdots\vert\gamma_T] \Vert_F$.
\label{pro:rhosimplex}
\end{lemma}
\begin{proof}
For the proof, we simply observe that every polytope $\mathcal{M}_{\Sigma}$ with $K$ vertices that is not a simplex is at most $(K-2)$-dimensional. We can then use $K-1$ vertices to construct a simplex which contains $\mathcal{M}_{\Sigma}$ which already proofs the proposition. Since $K\leq \dimD$, we are then left with one vertex we can use to increase the dimension of the polytope by constructing a $(K-1)$-dimensional simplex.
\end{proof}
For example, assume $K = 4$ and $\dimD \geq 4$. Then if $\mathcal{M}_{\Sigma}$ were not a ($3$-dimensional) simplex, it would at most be 2-dimensional, given by a polygon with $4$ vertices. However, for such a polygon, we could construct a triangle which contains this polygon while we could use the fourth vertex to reduce the projection error by increasing the polytope to a third dimension. As a consequence, solutions of $\eqref{eq:SPA1}$ will typically form a $(K-1)$-dimensional simplex (if not, one could add a constraint to \eqref{eq:SPA1}). We therefore always assume $\mathcal{M}_{\Sigma}$ to be a simplex whenever $K \leq D$.

With this, the function $\rho_\Sigma$ from Eq.~\eqref{eq:rhoK<D} is well-defined:
\begin{lemma}
Let $\Sigma \in \R^{\dimD \times K}$. The function $\rho_\Sigma$ defined as in Eq.~\eqref{eq:rhoK<D} is well-defined for all $\datapoint \in \R^{\dimD}$.
\label{lem:rhowelldefined}
\end{lemma}
\begin{proof}
By definition, the term $\Sigma \rho_\Sigma(\datapoint)$  is the closest point to $\datapoint$ in $\mathcal{M}_{\Sigma}$, so that it is the orthogonal projection of $\datapoint$ onto $\mathcal{M}_{\Sigma}$. This is always well-defined in Euclidean space. The question now is whether there exists $\tilde{\gamma}$ so that $\Sigma \gamma = \Sigma \tilde{\gamma}$. Since $\mathcal{M}_{\Sigma}$ is a simplex, the barycentric coordinates of the projection of $\datapoint$ onto $\mathcal{M}_{\Sigma}$ are unique~\cite{guessab}.
\end{proof}

$\rho_\Sigma$ can be seen as a linear coordinate transform of an orthogonal projection of points into $\R^\dimD$ onto a $K$-vertexed unit simplex whose elements are barycentric coordinates.

The definition of $\rho_\Sigma$ in Eq.~\eqref{eq:rhoK<D} for the case $K \leq \dimD$ is equivalent with the one in Eq.~\eqref{eq:rhoK>D}, since the set of minimizers of $\Vert \datapoint - \Sigma \gamma \Vert_2$ contains only one element. We could equivalently write it as $\rho_\Sigma(\datapoint,\bullet)$ and use the definition from Eq.~\eqref{eq:rhoK>D}.

\subsection{Linear approximations of nonlinear dynamics are typically insufficient to predict long-term behaviour}
\label{sec:App_SPAfixedpoints}
With \eqref{eq:SPA2}, we find a matrix $\Lambda$ that propagates projection coordinates linearly over time. The linearity plus the fact, that $\Lambda$ is a column-stochastic matrix, give severe restrictions on the capacities regarding long-term behaviour of the method: By the Perron-Frobenius Theorem~\cite{meyer}, a full-rank column-stochastic matrix must have eigenvalues $\lambda_i$ with $|\lambda_i| \leq 1$ and at least (and generally exactly) one real eigenvalue $\lambda_1 = 1$. Let us further assume no complex eigenvalues with $|\lambda_i| = 1$. We can write subsequent application of $\Lambda$ using its eigenvectors: Let $u_1,\dots,u_\clus$ be the right eigenvectors of $\Lambda$. Then they form a basis of $\R^\clus$. Let $\gamma_t$ have a representation with respect to these eigenvectors given by
\begin{equation*}
    \gamma_t = \sum\limits_{k=1}^\clus a_k u_k
\end{equation*}
where the $a_k$ are scalar-valued coefficients.
Then
\begin{equation*}
    \Lambda^i \gamma_t = \sum\limits_{k=1}^\clus a_k \Lambda^i u_k = \sum\limits_{k=1}^\clus a_k \lambda_k^i u_k.
\end{equation*}
All $\lambda_k$ where $|\lambda_j| < 1$ will vanish with increasing $i$, so that the sum converges to
\begin{equation*}
    \gamma^{\star} = a_1 u_1
\end{equation*}
(or, more generally, to $\sum\limits_{k:\lambda_k = 1} a_k u_k$).
Thus, in the long-term, a trajectory of barycentric coordinates created using the linear model must converge to a fixed point. Hence, also in the original state space a long-term prediction must converge to the fixed point $\datapoint^{\star} = \Smatrix \gamma^{\star}$.

If there exists complex eigenvalues with absolute value $1$, at least oscillating, but generally not truly linear dynamics can be constructed with a linear model.

\subsection{Proof that path affiliations are stochastic vectors}
\label{sec:App_pathaffStoch}
Let $\gamma_t,\dots, \gamma_{t-\memory+1}$ be stochastic vectors, i.e. all entries are at least $0$ and sum up to $1$. Then we can prove by induction on $\memory$ that the path affiliation vectors constructed by
\begin{equation*}
\begin{split}
\memop^\memory (\gamma_{t-1},\dots,\gamma_{t-\memory})_i &=  \gamma^{i_1}_{ t} \gamma^{i_2}_{t-1} \cdots \gamma^{i_\memory}_{t-\memory+1}\\
J(i) &= [i_1,\dots,i_\memory]
\end{split}
\end{equation*}
are stochastic vectors, too. $J$ again is an ordering of the tuples in $\lbrace 1,\dots,K\rbrace^\memory$.

Let $u,v$ be stochastic vectors of potentially different dimensions. Then
\begin{equation}
   \sum_{i} \Psi^2(u,v) = \sum\limits_{i_1, i_2} u_{i_1} v_{i_2} = \sum\limits_{i_1} u_{i_1} \sum\limits_{i_2} v_{i_2} = \sum\limits_{i_1} u_{i_1} = 1,
    \label{eq:pathAffilStoch}
\end{equation}
so that also
\begin{equation*}
    \sum\limits_{J(i) = [i_1, i_2]} \memop^2(\gamma_{t-1},\gamma_{t-2})_i =\sum\limits_{i_1, i_2} \gamma^{i_1}_{t-1} \gamma^{i_2}_{t-2} = \sum\limits_{i_1} \gamma^{i_1}_{t-1} \sum\limits_{i_2} \gamma^{i_2}_{t-2} = \sum\limits_{i_1} \gamma^{i_1}_{t-1} = 1.
\end{equation*}
For every $\memory > 1$, it holds

\begin{equation*}
   \memop^{\memory+1}(\gamma_{t-1},\dots,\gamma_{t-\memory},\gamma_{t-\memory-1}) =  \memop^2( \memop^\memory(\gamma_{t-1},\dots,\gamma_{t-\memory}), \gamma_{t-\memory-1})
\end{equation*}
(at least up to permutation of the entries). By the induction hypothesis that $\memop^\memory(\gamma_{t-1},\dots,\gamma_{t-\memory})$ is a stochastic vector, it follows from the induction start in Eq.~\eqref{eq:pathAffilStoch} that also $\memop^{\memory+1}( \memop^\memory(\gamma_{t-1},\dots,\gamma_{t-\memory}), \gamma_{t-\memory-1})$ is a stochastic vector.

\subsection{Proof that the image of the path affiliation function is a polytope}
\label{sec:App_psiimage}
Since each point of the form $\memop^\memory(\gamma_1,\dots,\gamma_\memory)$ is a $K^\memory$-dimensional stochastic vector, it lies inside a unit simplex with $K^\memory$ vertices. These vertices are the $K^\memory$-dimensional unit vectors. It remains to show that each point in such a polytope can be constructed by application of $\memop^\memory$ to a point $[\gamma_1,\dots,\gamma_\memory] \in \mathcal{M}_\Sigma^\memory$.

Let $[u_1,\dots,u_{K^\memory}] \in \mathcal{M}_\Sigma^\memory$. We then demand existence of $\gamma_1,\dots,\gamma_\memory$ so that
\begin{equation*}
    \begin{bmatrix}
    u_1\\
    \vdots\\
    u_{K^\memory}
    \end{bmatrix}
    =
    \begin{bmatrix}
    \gamma_1^1 \cdots \gamma_\memory^1\\
    \vdots\\
    \gamma_1^K \cdots \gamma_\memory^K
    \end{bmatrix}.
\end{equation*}
For this to hold requires several relations between the entries of $\gamma_1,\dots,\gamma_\memory$: The first entry of the above equation yields the value for $\gamma_1^1$ as
\begin{align*}
    \gamma_2^1\cdots\gamma_\memory^1 = \frac{u_1}{\gamma_1^1}, \quad \gamma_2^1\cdots \gamma_\memory^2 = \frac{u_2}{\gamma_1^1}, \quad \dots \quad \gamma_2^K\cdots \gamma_\memory^K = \frac{u_{K^{\memory-1}}}{\gamma_1^1}\\
    \Rightarrow \gamma_1^1 (\underbrace{\gamma_2^1\cdots\gamma_\memory^1  + \dots +\gamma_2^K\cdots \gamma_\memory^K}_{= \sum_j \Psi^{\memory-1}(\gamma_2,\dots,\gamma_\memory)_j = 1}) = u_1+ \dots +  u_{K^{\memory-1}}.
\end{align*}
This can be done analogously for all entries of $\gamma_1,\dots,\gamma_\memory$. We see that for all entries of $\gamma_1$, the terms needed to construct it are disjoint (for $\gamma_1^1$, $u_1,\dots,u_{K^{\memory-1}}$ are required, for $\gamma_1^2$ one needs $u_{K^{\memory-1}+1},\dots,u_{2K^{\memory-1}}$,...). Plus, by definition, the entries of $u$ sum to one so that this holds for $\gamma_1$ as desired. Analogously this can be observed for $\gamma_2,\dots,\gamma_\memory$, where again the entries from $u$ needed are disjoint between the entries of a $\gamma_i$.

In total, we can see that for each vector $u$ in the path affiliation polytope, we can directly construct $\gamma_1,\dots,\gamma_\memory \in \mathcal{M}_{\Sigma}$ so that the path affiliation function applied to them gives $u$. With this, the path affiliation function $\memop^\memory$ is surjective into the unit simplex with $K^\memory$ vertices. Since the entries of $\Psi^\memory$ always sum up to $1$ (see~\ref{sec:App_pathaffStoch}), no other points lie in its image.

\subsection{Proof that the path affiliation function $\memop^\memory$ is injective (Proposition~\ref{pro:psiinjective})}
\label{sec:App_psiinjective}
Let $\memop^\memory$ be defined as in Eq.~\eqref{eq:PSI_def} respectively Eq.~\eqref{eq:PSI_def_outerproduct} on all points $\gamma \in \R^K$ which fulfill $\sum_{i=1}^K \gamma^i = 1$ and $\gamma^i \geq 0$. The concatenation of injective functions is injective again, so it suffices to show that one outer product as defined in Eq.~\eqref{eq:PSI_def_outerproduct} is injective on the set of the barycentric coordinates:

Let $u,\tilde{u} \in \R^n$ and $v,\tilde{v} \in \R^m$ be given. We can observe that $u \otimes v = \tilde{u} \otimes \tilde{v}$ implies $\tilde{u}_1 = (u_1 v_1) / \tilde{v}_1$. Using this we can see that the demand $u_1 v_i = \tilde{u}_1 \tilde{v}_i$ yields for all $i = 2,\dots,K$
\begin{equation*}
     u_1 v_i = \tilde{u}_1 \tilde{v}_i = (u_1 v_1) \tilde{v}_i / \tilde{v}_1 \text{ so that } \tilde{v}_i = v_i (\tilde{v}_1 / v_1).
\end{equation*}
As a consequence, $\tilde{v}$ must be a scaled version of $v$ with an arbitrary scaling constant given by $\tilde{v}_1 / v_1$. This implies that $\tilde{u}$ can only be a scaled version of $u$. Taking into account the constraints on all vectors in the context of barycentric coordinates -- their entries must sum to $1$ -- , this means that $u = \tilde{u}$ and $v = \tilde{v}$.

Thus, $\memop^\memory$ is injective. This yields that the inverse of $\memop^\memory$ exists on its image.

\subsection{Proof of Proposition~\ref{pro:mSPAvsSPA2}}
\label{sec:AppProofLemma}
\begin{proof}
Let $\Lambda$ be a minimizer of
\begin{equation*}
\begin{split}
\Lambda &= \argmin_{\Lambda^*} \Vert [\gamma_2 \vert \cdots \vert\gamma_\timeT] - \Lambda^* [\gamma_1 \vert \cdots\vert\gamma_{\timeT-1}] \Vert_F, \\
\text{subject to,} & \\
& \Lambda \geq 0 \text{ and } \sum_{k=1}^K \Lambda_{k,\bullet} = 1. 
\end{split}
\end{equation*}
Choose $\memory \geq 1$. Then for all paths $i_1,\dots,i_\memory$, construct $\hat{\Lambda}^\memory \in \R^{K\times K^\memory}$ by setting $\hat{\Lambda}^\memory_{ji} = \Lambda_{j i_1}$ with $i$ denoting the index in the vector of path affiliations $\psi_t$ corresponding to the path $i_1,\dots,i_\memory$. In other words: for each sequence of $\memory$ indices, we simply set $\hat{\Lambda}^\memory_{ij}$ to the value in the $j$th column of $\Lambda$ that corresponds to the first (memoryless) entry of this sequence. Then since
\begin{equation*}
\begin{split}
    &\sum\limits_{i_2,\dots,i_\memory} \gamma_{t}^{i_1} \gamma_{t-1}^{i_2} \cdots \gamma_{t-\memory+1}^{i_\memory}\\
    =&\gamma_{t}^{i_1} \sum\limits_{i_2,\dots,i_\memory}  \gamma_{t-1}^{i_2} \cdots \gamma_{t-\memory+1}^{i_2} = \gamma_{t}^{i_1},
    \end{split}
\end{equation*}
it holds that
\begin{equation*}
\begin{split}
 (\hat{\Lambda}^\memory\psi^\memory_t)_j =& \sum\limits_{J(i) = [i_1,\dots,i_\memory]} \hat{\Lambda}^\memory_{ji} (\psi^\memory_t)_i\\
 =&  \sum\limits_{J(i) = [i_1,\dots,i_\memory]} \hat{\Lambda}^\memory_{ji} \gamma_{t}^{i_1} \gamma_{t-1}^{i_2} \dots \gamma_{t-\memory+1}^{i_\memory}\\
  =& \sum\limits_{i_1} \Lambda_{ji_1} \gamma_{t}^{i_1} \sum\limits_{i_2,\dots,i_\memory}\gamma_{t}^{i_1} \gamma_{t-1}^{i_2}  \cdots \gamma_{t-\memory+1}^{i_\memory}\\
  =& \sum\limits_{i_1}\Lambda_{j i_1} \gamma_{t}^{i_1} = (\Lambda \gamma_t)_j.
    \end{split}
\end{equation*}
Hence, a $\hat{\Lambda}^\memory$ constructed in this way for mSPA is equivalent to applying a solution of the standard \ref{eq:SPA2} problem. As a consequence, the training error of mSPA is always bounded from above by the training error of the standard \ref{eq:SPA2}.
\end{proof}

\subsection{Choice of $E$ in matrix $\Theta^\memory$}
\label{sec:App_choiceE}
The task is to extract the terms $\gamma_{t-1},\dots,\gamma_{t-\memory+1}$ from $\memop^\memory(\gamma_{t-1},\dots,\gamma_{t-\memory})$. For this aim, we need to realize that by summing over all entries of $\psi^\memory_{t-1}$ which contain a certain factor $\gamma_{t-k}^l$, we obtain $\gamma_{t-k}^l$ because the sum of all crossproducts (in the sense of the path affiliations) of two stochastic vectors is $1$.

Now let $\hat{\gamma}$ be a vector of the form
\begin{equation*}
    \hat{\gamma} = \begin{pmatrix}
    \gamma_{\memory}\\
    \vdots\\
    \gamma_{1}
    \end{pmatrix} = \begin{pmatrix}
    \hat{\gamma}_{1}\\
    \vdots\\
    \hat{\gamma}_{K\memory}
    \end{pmatrix} \in \R^{K\memory}
\end{equation*}
where all $\gamma_i$ are stochastic vectors in $\R^K$. We want a $(\memory-1)K \times K^\memory$ matrix $E$ that fulfills
\begin{equation*}
    E\memop^\memory(\gamma_{t-1},\dots,\gamma_{t-\memory}) = \begin{pmatrix}
    \gamma_{t-1}\\
    \vdots\\
    \gamma_{t-\memory+1}
    \end{pmatrix}
    \in \R^{(\memory-1)K}.
\end{equation*}
Let
\begin{equation*}
    E_{ij} =
    \begin{cases}
        1 & \text{if } (j \in i)\; (\star)  \\
        0 & \, \text{else}
    \end{cases}
\end{equation*}
By $(\star)$ we mean that $\hat{\gamma}_i$ is one of the factors in $\Psi(\hat{\gamma})_j$.

Let for an arbitrary $i$ $\hat{\gamma}_i = \gamma_{t-k}^l$. Then
\begin{equation*}
\begin{split}
    (E \memop^\memory(\gamma_{t-1},\dots,\gamma_{t-\memory}))_i = &\sum\limits_{j = 1}^{K^\memory} E_{ij} \memop^\memory(\gamma_{t-1},\dots,\gamma_{t-\memory})_i\\
    =& \sum\limits_{J(i) = [i_1,\dots,i_\memory], i_k = l} \gamma_{t-1}^{i_1}\cdots \gamma_{t-\memory}^{i_\memory}\\
    =& \gamma_{t-k}^l \underbrace{\sum\limits_{i_1,\dots,i_\memory} \gamma_{t-1}^{i_1} \cdots ^\star{\gamma}_{t-k}^l \cdots \gamma_{t-\memory}^{i_\memory}}_{=1} = \gamma_{t-k}^l
    \end{split}
\end{equation*}
where $^\star{\gamma}_{t-k}^l$ means that it is omitted in the product. $E$ uses the fact that summing over all products of terms in which one specific term $\gamma_{t-k}^l$ occurs is equal to $\gamma_{t-k}^l$. With this, this choice of $E$ fulfills the demand we set on it.

\subsection{The embedded dynamics are quadratic}
\label{sec:App_quad}

As the memory depth $\memory$ is fixed, we will omit it in the notation of the path affiliations. Thus, we write~$\psi_{t-1}$ and enumerate its elements by multiindices,
\[
\psi_{t-1}^{i_1\ldots i_\memory} = \gamma_{t-1}^{i_1} \gamma_{t-2}^{i_2} \cdots \gamma_{t-\memory}^{i_\memory}, \quad i_1,\ldots,i_M \in \{1,\ldots,K\}.
\]
Similarly, we enumerate the columns of $\hat{\Lambda}^\memory$ by multiindices, and we can rephrase \eqref{eq:mSPAsystem} to
\[
\gamma_t^i = \sum_{j_1,\ldots,j_\memory=1}^K \hat{\Lambda}^\memory_{i, j_1\ldots j_\memory} \, \psi_{t-1}^{j_1\ldots j_\memory}. 
\]
We now note that for arbitrary $i_1,\ldots,i_{\memory-1} \in \{1,\ldots,K\}$,
\[
\gamma_{t-1}^{i_1} \cdot \ldots \cdot \gamma_{t-\memory+1}^{i_{\memory-1}} = \sum_{j=1}^K \gamma_{t-1}^{i_1} \cdot \ldots \cdot \gamma_{t-\memory+1}^{i_{\memory-1}} \cdot \gamma_{t-\memory}^j = \sum_{j=1}^K \psi_{t-1}^{i_1 \ldots i_{\memory-1}\, j}.
\]
It then follows by the above that
\begin{align*}
    \psi_t^{i\, i_1\ldots i_{\memory-1}} &= \gamma_t^i \gamma_{t-1}^{i_1} \cdot \ldots \cdot \gamma_{t-\memory+1}^{i_{\memory-1}} \\
    &= \gamma_{t-1}^{i_1} \cdot \ldots \cdot \gamma_{t-\memory+1}^{i_{\memory-1}} \sum_{j_1,\ldots,j_\memory=1}^K \hat{\Lambda}^\memory_{i, j_1\ldots j_\memory} \psi_{t-1}^{j_1\ldots j_\memory}\\
    &= \sum_{j=1}^K \psi_{t-1}^{i_1 \ldots i_{\memory-1}\,j} \sum_{j_1,\ldots,j_\memory=1}^K \hat{\Lambda}^\memory_{i, j_1\ldots j_\memory} \psi_{t-1}^{j_1\ldots j_\memory} \\
    &= \sum_{j,j_1,\ldots,j_\memory=1}^K \hat{\Lambda}^\memory_{i, j_1\ldots j_\memory} \psi_{t-1}^{j_1\ldots j_\memory} \psi_{t-1}^{i_1 \ldots i_{\memory-1}\,j}\,. 
\end{align*}
In summary, the mapping $\psi_{t-1} \mapsto \psi_t$ is componentwise a quadratic polynomial.

\subsection{Numerical properties of mSPA}
\begin{remark}
The numerical complexity of \eqref{eq:SPA1} scales cubically with $K$ and linearly with $\dimD$ and $T$. Detailed information on the numerical and memory complexity of solving \eqref{eq:SPA1} can be found in \cite{spaPaper}. The \eqref{eq:SPA2} problem is solved by reformulating it into a quadratic programming problem which is then solved by a Spectral Projected Gradient Algorithm~\cite{birgin} (see Matlab code in \url{https://github.com/SusanneGerber/SPA}).
\end{remark}

\begin{remark}
For all examples, we normalized the data by a linear shifting and scaling so that in each coordinate, the minimum and maximum values were $-1$ and $+1$ to give equal weight to all coordinates.
\end{remark}

\begin{remark}
When implementing Algorithm \ref{alg:mSPAscheme}, we suggest to normalize the values of $\gamma_t$ in each step so that the entries sum up to $1$. This is the case theoretically but numerical errors can have severe effects on the simulation of the dynamics.
\end{remark}

\section{Takens' Theorem and definitions needed for the extension Theorem~\ref{thm:robinsonTakens}}
\label{sec:App_Takens}
The original delay embedding Theorem of Takens reads
\begin{theorem}[Takens, 1981~\cite{takens}]
Let $\mathcal{M}$ be a smooth, compact, $d$-dimensional manifold. The set of pairs $(\textup{\gener},\textup{\observ})$ for which the delay embedding $\Phi$ is an embedding is generic in $\mathcal{D}^r(\mathcal{M})\times \mathcal{C}^r(\mathcal{M},\R)$ if $\memory > 2d$ for $r \geq 2$.
\label{thm:takens}
\end{theorem}
$\mathcal{D}^r(\mathcal{M})$ denotes the set of $r$-times continuously differentiable diffeomorphisms mapping from $\mathcal{M}$ to $\mathcal{M}$, i.e., the set of dynamics $\gener$ on $\mathcal{M}$. Generic means that the set of pairs of functions for which this is the case, is open and dense in the $\mathcal{C}^1$ topology of functions. Takens originally showed this for $r =  2$ but it was shown in \cite{huke} that it holds for $r = 1$.

The fact that $\Phi$ is an embedding means that topological properties such as Lyapunov exponents or nature and number of fixed points of $\gener$ are preserved.

\subsection{Prevalence}
The term prevalence defines a notion for almost every in infinite-dimensional spaces. 
\begin{definition}[Prevalent~\cite{sauer}]
A Borel subset $S$ of a normed linear space $V$ is prevalent if there is a finite-dimensional
subspace $E$ of $V$ such that for each $v \in V$, $v + e$ belongs to $S$ for Lebesgue almost every $e \in E$.
\end{definition}

\subsection{Thickness exponent}
The thickness exponent measures the degree to which subset $X$ of a Banach space $B$ can be approximated by linear subspaces of $B$.
\begin{definition}[Thickness exponent~\cite{robinson}]
Let $\epsilon_B(X, n)$ denote the minimal distance between $X$ and any $n$-dimensional linear subspace of $B$. Then
\begin{equation*}
    \tau(X,B) := \lim_{n\rightarrow \infty} \frac{-\log(n)}{\log (\epsilon_B(X,n))}
\end{equation*}
\end{definition}

In \cite{robinson}, a condition is given for which $ \tau(X,B) = 0$, i.e., $ \lim_{n\rightarrow \infty} \epsilon_B(X,n) = 0$. The condition refers to the smoothness of $X$.

\subsection{Box-counting dimension}
The box-counting dimension is a measure for the dimensionality of sets.
\begin{definition}[Box-counting dimension~\cite{robinson}]
The box-counting dimension of a compact set $A$ is defined as
\begin{equation*}
boxdim(A) = \lim\limits_{\varepsilon \rightarrow 0} \frac{-\log( N_A(\varepsilon))}{\log(\varepsilon)},
\end{equation*}
where $N_A(\varepsilon)$ is the number of cubes with edge length $\varepsilon$ needed to cover $A$.
\end{definition}
The box-counting dimension measures the speed in which the number of boxes needed to cover a set rises when decreasing their size. If $A$ is a $d$-dimensional manifold, then $N_A(\varepsilon)$ is porportional to $ \varepsilon^{-d}$ so that $boxdim(A) = d$. The box-counting dimension is generally a non-integer value for sets which are not a manifold.

\section{Technical details on numerical experiments}
\subsection{Definition of error measures}
\label{sec:App_errors}
We evaluate the forecasting capacity of the used methods with the following error metrics.
\begin{definition}[\gammaerror]
The \tgammaerror between points $\gamma, \tilde{\gamma}$ in a unit simplex $\mathcal{M}_{\Sigma}$ is defined as
\begin{equation}
    \mathcal{E}_{\Sigma}(\gamma,\tilde{\gamma}) := \Vert \gamma - \tilde{\gamma} \Vert_2.
\end{equation}
\end{definition}

\begin{definition}[\truespaceerror]
The \ttruespaceerror between points $\datapoint,\tilde{\datapoint} \in \R^\dimD$ is defined as
\begin{equation}
    \mathcal{E}_{\dimD}(\datapoint,\tilde{\datapoint}) := \Vert \datapoint - \tilde{\datapoint} \Vert_2.
\end{equation}
where $\tilde{\datapoint} = \Sigma \tilde{\gamma}$ for a $\gamma$ which comes from a prediction with SPA or mSPA.
\end{definition}

\begin{definition}[Average $k$-step \gammaerror]
For $k$-step forecasts, we create trajectories of length $k$ starting at $T$ different values $\gamma_{t_1},\dots,\gamma_{t_T}$ to generate values $\tilde{\gamma}_{t_1+k},\dots,\tilde{\gamma}_{t_T+k}$. Given reference values $\gamma_{t_1+k},\dots,\gamma_{t_T+k}$, we define the average \tgammaerror as
\begin{equation}
    \bar{\mathcal{E}}_{\Sigma} := \frac{1}{T} \sum\limits_{i=1}^T \mathcal{E}_{\Sigma}(\gamma_{t_i+k},\tilde{\gamma}_{t_i+k}).
\end{equation}
\end{definition}
Analogously for the true space we define:
\begin{definition}[Average $k$-step \truespaceerror]
For $k$-step forecasts, we create trajectories of length $k$ starting at $T$ different values $\datapoint_{t_1},\dots,\datapoint_{t_T}$ to generate values $\tilde{\datapoint}_{t_1+k},\dots,\tilde{\datapoint}_{t_T+k}$. Given reference values $\datapoint_{t_1+k},\dots,\datapoint_{t_T+k}$, we define the average \ttruespaceerror as
\begin{equation}
    \bar{\mathcal{E}}_{\dimD} := \frac{1}{T} \sum\limits_{i=1}^T \mathcal{E}_{\dimD}(\datapoint_{t_i+k},\tilde{\datapoint}_{t_i+k}).
\end{equation}
\end{definition}

\subsection{Technical details on Examples \ref{exa:KSlowdim} and \ref{exa:KSlifting}}
\label{sec:App_techsKS}
\begin{remark}
\textbf{Data division:} We divided the data into $T_{train} = 3000$ training points and $T_{test} = 1000$ test points, each separated by a time of $0.001$. We learned propagators for a time step length of $\tau = 0.01$.
\end{remark}
\begin{remark}
\textbf{Error computation in Figure~\ref{fig:KS_gammaprediction_K2_tau001}:} To measure the quality of mSPA predictions, we evaluated the forecasting error and made long-term predictions. The $k$-step forecasting error in Figure~\ref{fig:KS_gammaprediction_K2_tau001} is the average $k$-step \tgammaerror between the predicted and true values of $\gamma$ from the test data for all starting points between time steps $\memory,\dots,T_{test}-\memory$.
\end{remark}
\begin{remark}
\textbf{Basis functions in functioning SINDy model in Example~\ref{exa:KSlifting}:} The basis functions in Figure~\ref{fig:KS_allmethods_Comparison} \textbf{c.} were given by all quadratic monomials of $u^{(1)},\dots,u^{(100)}$ and the functions $\sin(k u^{(i)}), \cos(k u^{(i)})$ for $i = 1,\dots,100$, $k = 1,\dots,10$. If linear and constant terms were included, the model determined produced diverging trajectories. The same held when omitting some or all of the trigonometric functions or using monomials up to degree $3$.
\end{remark}

\begin{remark}
From two trajectories $X = [x_0,\dots,x_T]$ and $\tilde{X} = [\tilde{x}_0,\dots,\tilde{x}_T]$, we defien their Hausdorff distance as
\begin{equation}
\mathcal{H}(X,\tilde{X}) := \max( \max_{x\in {X}} \min_{\tilde{x} \in \tilde{X}} \Vert x - \tilde{x}\Vert_2, \max_{\tilde{x}\in X} \min_{x \in X} \Vert x - \tilde{x}\Vert_2 ).
\label{eq:hausdorff}
\end{equation}
\end{remark}

\section{Approximating the mapping on the path affiliations by a neural network}
\label{sec:App_NN}
In mSPA we approximate the function \textbf{v} from Eq.~\eqref{eq:memorydynamics} by a column-stochastic matrix. We further have tried to use a neural network instead to obtain a more precise approximation.

We have constructed networks of the following structure. Denote $h^{(0)} = \psi_{t-1}$. For a fixed memory depth $\memory$, let $i = 1,\dots,\memory-1$. The $i$-th layer $h^{(i)}$ is given by
\begin{equation}
    \begin{split}
        h^{(i)} &= A^{(i)} (W^{(i-1)} h^{(i-1)} + b^{(i-1)})\\
        W^{(i-1)} &\in \R^{K^{\memory-i}\times K^{\memory-i+1}}, b^{(i-1)} \in \R^{K^{\memory-i}},
    \end{split}
\end{equation}
and $A^{(i)}$ is a nonlinear activation function, chosen either as a Leaky ReLU (Rectified Linear Unit) or a Softmax function, defined as
\begin{equation}
    \text{Leaky ReLU}(\datapoint_i) = \begin{cases}
0.1, \datapoint_i \leq 0.1\\
\datapoint_i, \text{ otherwise}
\end{cases}
\quad
softmax(\datapoint)_i = \exp(\datapoint_i) / \sum_{j=1} \exp(\datapoint_j).
\end{equation}
The network thus reduces the dimension of hidden states by powers of $K$ and consists of $\memory-1$ layers, so that the last layer is $K$-dimensional. The output is defined as
\begin{equation}
    y_t = softmax(h^{(\memory-1)}),
\end{equation}
projecting it to barycentric coordinates.


We were able to get a close reconstruction of the Lorenz-96 and Chua attractors in barycentric coordinates with $\memory = 3$ and $M=5$, respectively (Figure~\ref{fig:LorenzChua_NN_Full}).
\begin{figure}[ht!]
\centering
\includegraphics[width=\textwidth]{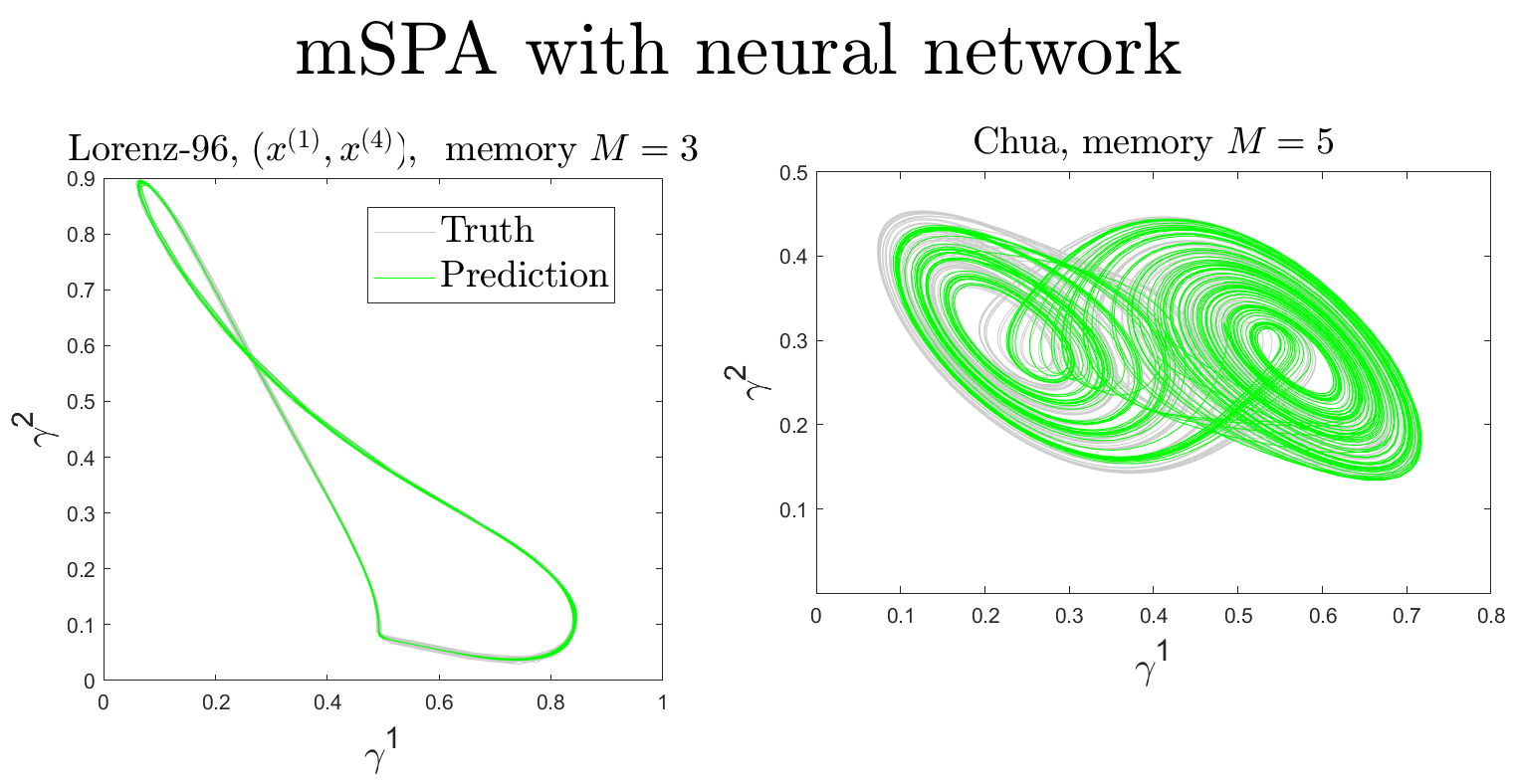}
\caption{Long-term predictions with a neural network to map from path affiliation to subsequent barycentric coordinate for Lorenz-96 (($\datapoint^{(1)},\datapoint^{(4)})$-coordinates) and the inner attractor of the Chua circuit.}
\label{fig:LorenzChua_NN_Full}
\end{figure}

\section{mSPA on the Chua circuit}
\label{sec:App_chua}
\begin{example}
\textbf{The Chua circuit} A Chua circuit is a model of an electric circuit with chaotic behaviour. It exists in several variants (see the original publications in \cite{chua,chuaWu} and an application of a numerical method to a variant in \cite{gelss}). We consider the following three-dimensional one:
\begin{equation}
    \begin{split}
        \dot{\datapoint}^{(1)} &= \alpha \datapoint^{(2)} - \mu_0 \datapoint^{(1)} - \frac{\mu_1}{3} (\datapoint^{(1)})^3\\
         \dot{\datapoint}^{(2)} &= \datapoint^{(1)} - \datapoint^{(2)} +\datapoint^{(3)}\\
         \dot{\datapoint}^{(3)} &= -\beta \datapoint^{(2)}\\
         \alpha = 18,\quad \beta &= 33,\quad \mu_0 = -0.2,\quad \mu_1 = 0.01
    \end{split}
    \label{eq:chua}
\end{equation}

With these equations and parameters, the system has an attractor with two connected components
\begin{figure}[ht]
\includegraphics[width = \textwidth]{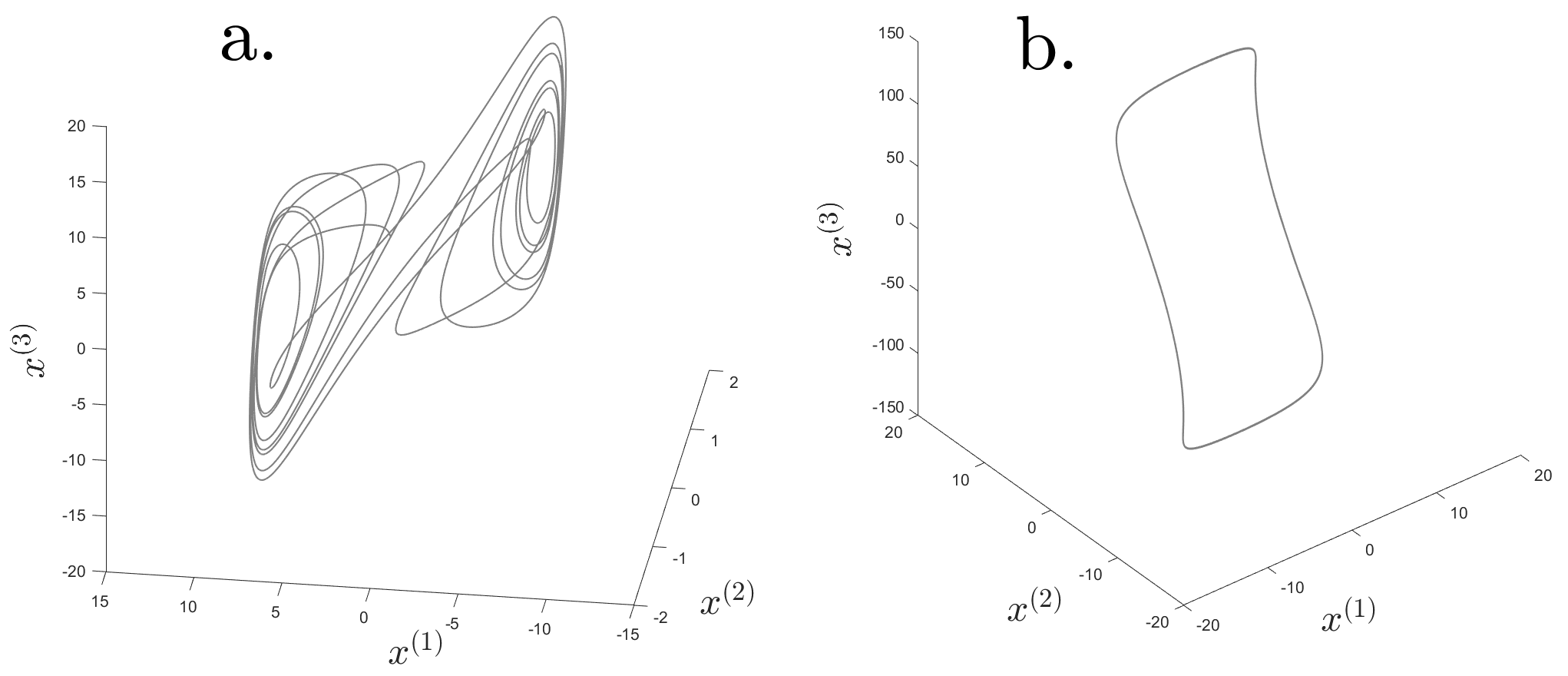}
\caption{Trajectories of the two connected components of the attractor of the Chua circuit from Eq.~\eqref{eq:chua}. \textbf{a.} Inner attractor consisting of switching between two lobes. \textbf{b.} Outer periodic cycle. Note the difference in the axes' scales between a.\ and b.}
\label{fig:chuaAttractors}
\end{figure}
(see Figure~\ref{fig:chuaAttractors}): one limit cycle and a more complex set (supporting complicated dynamical behaviour) which consists of two lobes, similarly to the Lorenz-63 attractor. We simulate two trajectories of $20000$ steps each starting close to each attractor with a step size of $h = 0.001$ with the RK4-scheme. We then solve \ref{eq:SPA1} with $K = 3$ and $K' = 4$ to construct a learning and a lifting polytope. For the mSPA model, we utilize memory here differently than in the previous examples due to the small step size. While previously we tried to predict $\datapoint_{t}$ on the basis of $\datapoint_{t-\tau},\dots,\datapoint_{t-\memory \tau}$, we now use a value for $\tau$ as the \textit{memory time step} and introduce the forward step size given by $\Delta t$ to predict $\datapoint_{t+\Delta t}$ with $\datapoint_{t},\datapoint_{t-\tau},\dots,\datapoint_{t-(\memory-1) \tau}$. Given the small step size $\Delta t$, the delay-coordinate map would project all points close to the set $\{(\datapoint_1,\dots,\datapoint_\memory) \in \R^{\memory\dimD} :x_1 = x_2 = \ldots = x_{\memory\dimD}\}$, strongly distorting the geometry of the original attractor (see Figure~\ref{fig:chua_memoryTimeStep_Full}). The same then holds for the delay-coordinate map of projections to a polytope. We therefore choose a bigger memory time step $\tau$ so that the delay-coordinate map unfolds the geometry of the attractor well but still make small forward time steps to preserve the accuracy of the approximated dynamics.
\begin{figure}[ht]
\includegraphics[width = \textwidth]{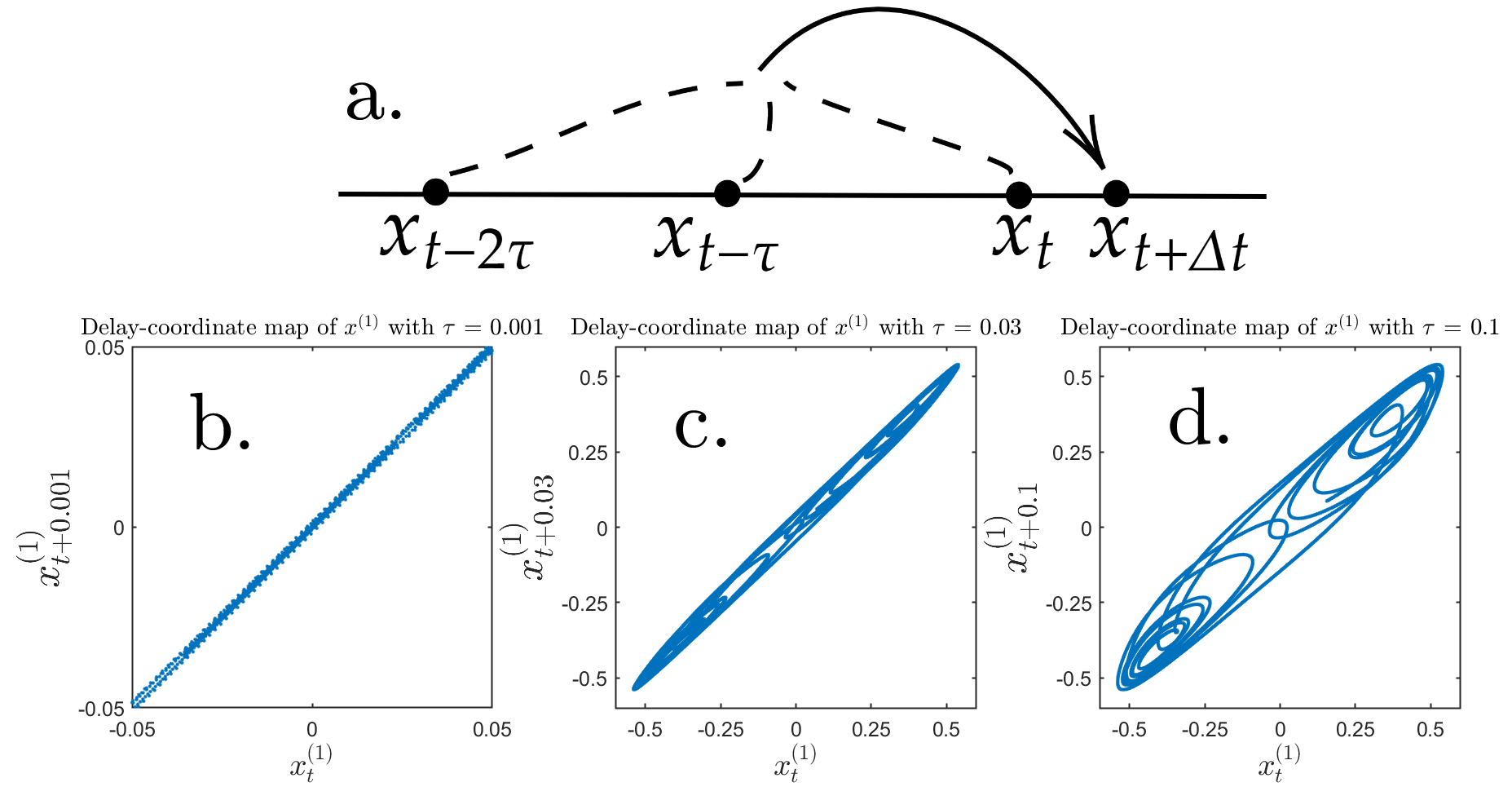}
\caption{\textbf{a.} Concept of using points separated by the memory time step $\tau$ to construct predictions using a smaller step size $\Delta t$. \textbf{b.}--\textbf{c.} Images of the 2-delay-coordinate map of $\datapoint^{(1)}$ with different values for $\tau$. For small $\tau$, the image of the 2-delay-coordinate map is close to the diagonal in $\R^2$ while for larger $\tau$ its shape comes closer to the one of the full Chua circuit.}
\label{fig:chua_memoryTimeStep_Full}
\end{figure}

We use $\Delta t = h = 0.001$ and $\tau = 0.03$, i.e., 30-fold the length of the prediction time step, and use $\memory = 7$ for the mSPA propagator. Clearly, a memory depth of $7$ is far greater than should be needed as suggested by Takens' Theorem, but it seems to be necessary for mSPA to capture the nonlinearity of the system. Since the Chua circuit is a chaotic system, we do not consider the short-term prediction error but compare the attractors of the true and predicted trajectories. In Figure~\ref{fig:chuaFullPlot} are illustrated the resulting long-term predictions of mSPA for both the inner attractor and the outer cycle. The outer cycle is easier to reconstruct and mSPA generates accurate predictions. For the inner attractor mSPA does not derive a quantitatively close reconstruction but does generate a qualitatively correct prediction maintaining the characteristic switching dynamics between the two lobes.

However, at least when using data from the inner and the outer attractors for the mSPA training, it proved difficult to find a polytope and a memory time step which enabled mSPA to produce this result. Many polytopes defined by different vertices, including the ones used here but scaled with a factor and different memory time steps, made mSPA fail in reconstructing the switching behaviour of the inner attractor. We always required a memory depth of $7$ to reconstruct the switching behaviour when we trained with trajectories converging to both the inner and outer attractors. Nevertheless, when we trained only on the inner attractor and with a memory depth of $6$, mSPA did recreate the switching between the lobes with most of the polytopes used (different solutions of \eqref{eq:SPA1}). In the barycentric coordinates, the shape of the inner attractor was well reconstructed, while the lifting to the full state space perturbed the curvature slightly. We further saw that choosing polytopes that are tight around the projected data of the inner attractor are usually more suitable. Therefore, given the chaotic nature of the system, although the mSPA predictions are not exact, we are still encouraged by the facts that an mSPA model could be found which qualitatively recreates the dynamics of both the outer and inner attractors, and that for the inner attractor only we can find models with a lower memory depth for most of the polytopes we used.

It is hence desirable to find, if not a rigorous mathematical theory, then at least an intuition on how to choose parameters of the method to produce optimal results. To this end, we made the following theoretical and numerical observations. 


\begin{figure}[ht]
\includegraphics[width = \textwidth]{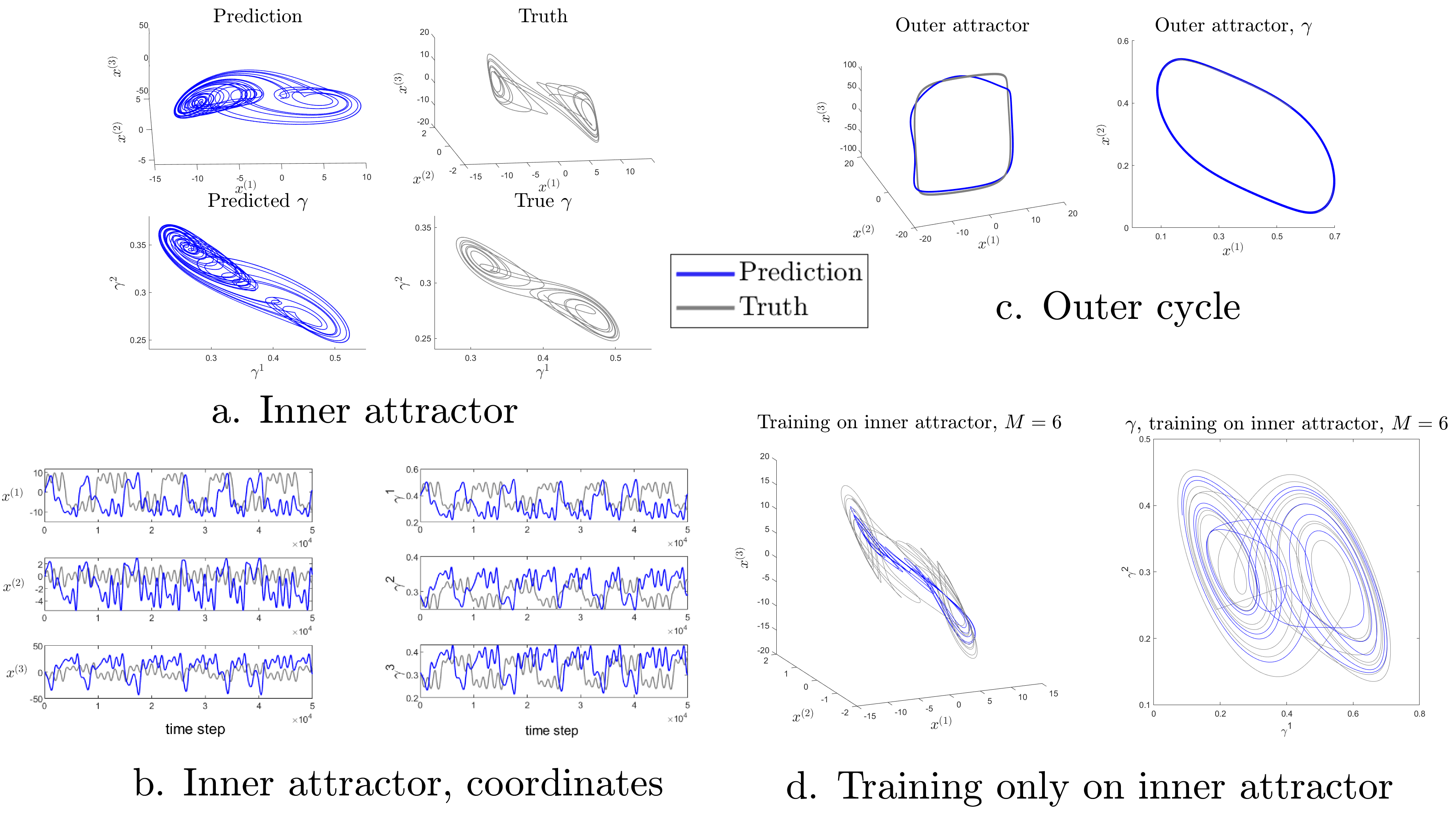}
\caption{Long-term predictions with the same mSPA model for the Chua circuit, starting on both attractors. We can see that the outer cycle can be accurately reconstructed with $\memory=7$ (\textbf{c.}). The inner attractor proves more difficult due to its chaotic behaviour. mSPA is however able to reproduce the lobe-switching behaviour (\textbf{a.} and \textbf{b.}) while the the reconstructed attractor is bigger than the original one and not as curved. Training only on the inner attractor admitted a closer reconstruction (\textbf{d.}) and already with $\memory = 6$.}
\label{fig:chuaFullPlot}
\end{figure}
\end{example}

\paragraph{Numerical observations}
While studying the Chua circuit, we made the following observation: while a memory depth of $7$ was necessary to reconstruct the switching in the inner attractor when considering both the inner attractor and outer cycle for training, we could succeed with $\memory = 6$ when considering only the inner attractor, hence choosing a smaller polytope and deriving a different matrix for $\hat{\Lambda}^\memory$. We chose different memory lags and used the same polytope as above but varied its size by scaling it with factors between $0.8$ and $1.3$. For each solution for the resulting \eqref{eq:mSPA2} problem, we computed a trajectory of $20000$ time steps and compared the distance between both attractors using the Hausdorff distance (see Appendix~\ref{sec:App_techsKS}). The result is depicted in Figure~\ref{fig:parameterStudy_hausmeas_chua}. We can see a niche which gives the optimal parameter setting for this example. The memory lag is best selected as $25$, close to the value of $30$ selected above. Figure~\ref{fig:chua_hausdorffMemorylag} shows the dependence on the memory lag for the scaling factor $1$ in more detail. As is consistent with the theoretical observations, there also exists a niche for the polytope size. Its optimal scaling factor is $1.05$, indicating that the polytope resulting from \ref{eq:SPA1} was already close to a good choice.
\begin{figure}
\includegraphics[width = \textwidth]{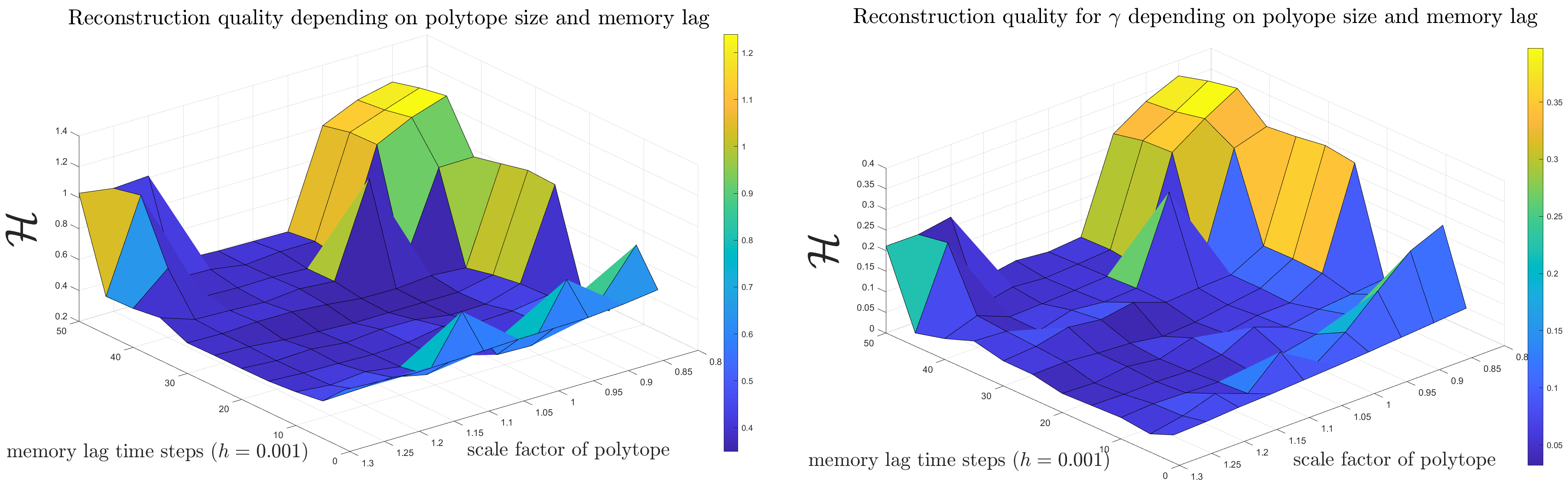}
\caption{Hausdorff distance between true trajectory (full space and barycentric coordinates) and mSPA prediction depending on parameters memory lag and size of polytope.}
\label{fig:parameterStudy_hausmeas_chua}
\end{figure}
\begin{figure}[ht]
\includegraphics[width = \textwidth]{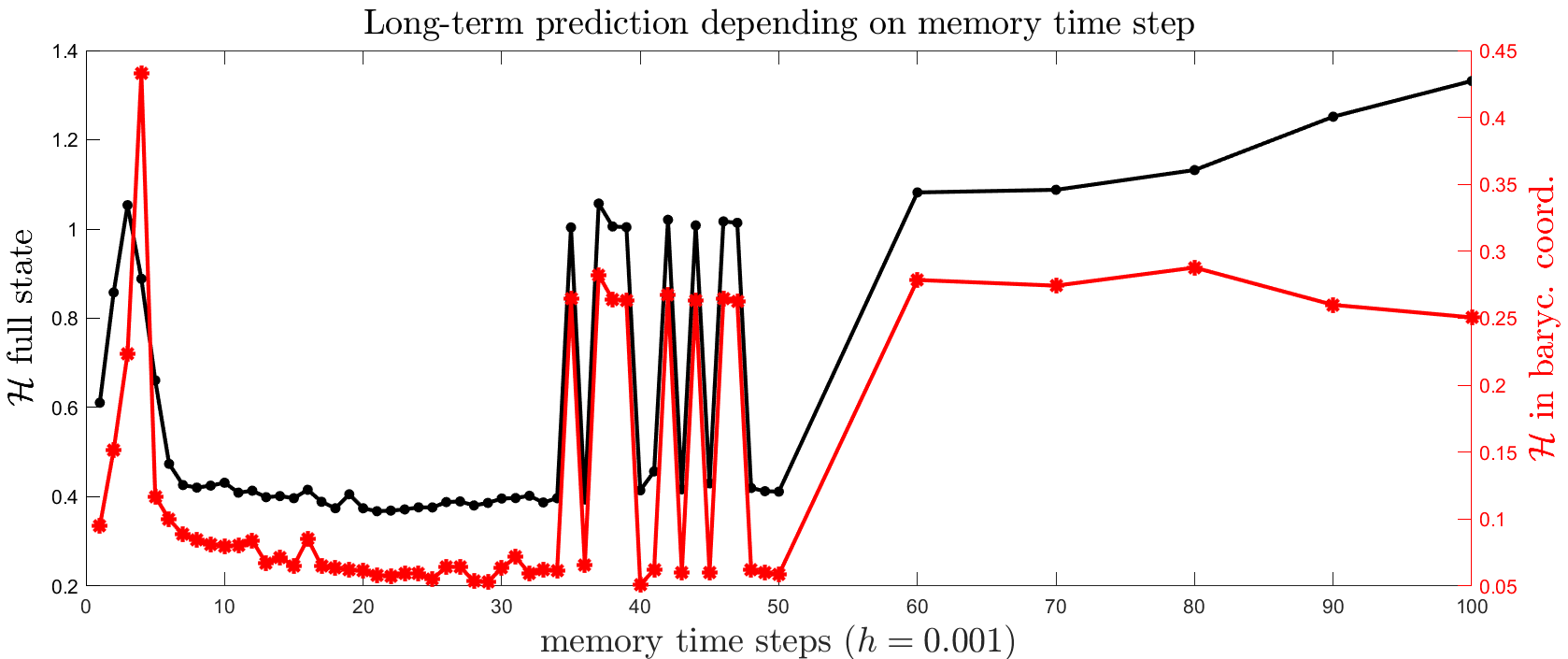}
\caption{Hausdorff distance $\mathcal{H}$ between reconstructed and true attractors dependent on the memory lag. Black: full state space. Red: in barycentric coordinates.}
\label{fig:chua_hausdorffMemorylag}
\end{figure}

\section{mSPA in connection to Koopman operator based methods}
\label{sec:App_Koopman}
We previously argued that mSPA is conceptually different than common regression methods since we project to a lower-dimensional manifold instead of directly transforming data into a higher-dimensional space. We do, however, want to point out several relations to such methods.

Identifying governing equations for the evolution of observables over time is addressed by different methods. Their common mathematical ground is the approximation of the action of the Koopman operator, $\mathcal{K}$, which is defined as propagating observables, so that~$\mathcal{K}\circ \observ = \observ \circ \gener$. The Koopman operator is an infinite-dimensional linear operator~\cite{klusKoopman}, for which different methods attempt to find a finite-dimensional approximation. To this end, the observable is typically transformed into a higher-dimensional space through application of scalar-valued basis functions $\observ_1,\dots,\observ_L$ which are then mapped back into the observation space by multiplication of a coefficient matrix $B$, i.e.,
\begin{equation}
    \mathcal{K}\circ \observ(\datapoint) \approx B [\observ_1,\dots,\observ_L]^T(\datapoint).
\end{equation}
In this light, we can observe the following relations to several common system identification techniques.
\paragraph{Sparse Identification of Nonlinear Dynamics (SINDy)}
The SINDy problems reads,
\begin{equation}
\Theta = \argmin_{\Theta'\in \R^{\dimD\times L}} \Vert \begin{bmatrix}
\observ(\datapoint_1),\dots,\observ(\datapoint_T)
\end{bmatrix}
- \Theta' \begin{bmatrix}
\observ_1(\datapoint_0) & \dots & \observ_1(\datapoint_{T-1})\\
\vdots & & \vdots\\
\observ_L(\datapoint_0) & \dots & \observ_L(\datapoint_{T-1})
\end{bmatrix}
\Vert_F + c \sum_{i,j=1}^{d,L} | \Theta'_{i,j} |.
\label{eq:SINDyminimizer}
\end{equation}
for a $c\geq 0$.
In SINDy~\cite{sindy,sindyc}, one tries to find $B$ by solving a regression problem between $[\observ_1,\dots,\observ_L]$ and $\mathcal{K}\circ\observ$ with the option to add a sparsity constraint. It has proven to an effective method in discovering the rules behind nonlinear dynamics in recent work (e.g.,~\cite{havok,klus2020}). mSPA can be seen as a special case of SINDy with memory terms, a special choice of basis functions -- the path affiliations -- and additional constraints on $\hat{\Lambda}$. Due to these choices, one is able to generate stable memory-requiring dynamics, such as in Example~\ref{exa:overlap}. We have seen in the previous examples that SINDy is prone to find models which are unstable. Further, one has to specify the basis functions in $\eta$ beforehand which either requires an intuition about the problem or a high number of basis functions. The latter can quickly yield an ill-conditioned nonlinear regression problem~\cite{freund}, so that the ensuing model might produce unstable dynamics. In mSPA we have a specific set of basis functions for which only three parameters, $K,K'$ and the memory depth $\memory$ have to be specified.

In \cite{lusch,champion}, two approaches to modelling dynamics are presented that are similar to the Learning and Lifting approach of mSPA. Both are based on the simultaneous learning of dimension reduction and dynamics identification in a single optimization problem. In \cite{lusch}, eigenfunctions of the Koopman operator of a dynamical system are found by using an autoencoder neural network. In \cite{champion}, a framework is presented in which data are compressed, propagated in time and lifted back using basis functions. In contrast to mSPA, both approaches do not use memory to counter the loss of information induced by dimension reduction. Rather, they rely on the data admitting a low dimensional representation which is sufficient to learn dynamics and make the lifting. In future work, mSPA could be modified in the light of this approach to do the dimension reduction, learning and lifting in one step to find the optimal polytope for dynamics reconstruction.

\paragraph{Extended Dynamic Mode Decomposition (EDMD)}
The EDMD problem reads,
\begin{equation}
\textbf{K} = \argmin_{\textbf{K}'\in \R^{L\times L}} \Vert \begin{bmatrix}
\observ_1(\datapoint_1) & \dots & \observ_1(\datapoint_{T})\\
\vdots & & \vdots\\
\observ_L(\datapoint_1) & \dots & \observ_L(\datapoint_{T})
\end{bmatrix}
- \textbf{K}' \begin{bmatrix}
\observ_1(\datapoint_0) & \dots & \observ_1(\datapoint_{T-1})\\
\vdots & & \vdots\\
\observ_L(\datapoint_0) & \dots & \observ_L(\datapoint_{T-1})
\end{bmatrix}
\Vert_F.
\label{eq:EDMDminimizer}
\end{equation}
Conceptually similar is EDMD~\cite{williams,liEDMD} which focuses on determining a linear mapping between $[\observ_1,\dots,\observ_L]^T$ and $[\mathcal{K}\circ\observ_1,\dots,\mathcal{K}\circ\observ_L]^T$ with the aim of uncovering eigenvalues and -functions of the Koopman operator. A special case is Hankel Dynamic Mode Decomposition (Hankel-DMD~\cite{arbabi}) where one uses memory by choosing $[\observ_1,\dots,\observ_L]$ as $[\observ,\mathcal{K}^{-1}\circ \observ,\dots,\mathcal{K}^{-\memory+1}\circ \observ]$ for this aim. In this light, mSPA is conceptually similar to Hankel-DMD, given that both use memory terms, with the observable $\observ$ then being the projection to the polytope $\mathcal{M}_\Sigma$. mSPA differs from Hankel-DMD in the facts that it does not apply a linear mapping to the memory terms directly but to the path affiliations and further aims at propagating dynamics rather than identifying characteristics of the Koopman operator. Of course, in mSPA the linear mapping is constrained to be a column-stochastic matrix.

\paragraph{Linear Autoregressive models (AR)} 
The AR problem reads,
\begin{equation}
\theta = \argmin_{\theta'\in \R^{\dimD\times \dimD\memory}} \Vert \begin{bmatrix}
\observ(\datapoint_\memory) & \dots & \observ(\datapoint_{T})\\
\end{bmatrix}
- \theta' \begin{bmatrix}
\observ(\datapoint_{\memory-1}) & \dots & \observ(\datapoint_{T-1})\\
\vdots & & \vdots\\
\observ(\datapoint_{0}) & \dots & \observ(\datapoint_{T-\memory})
\end{bmatrix}
\Vert_F.
\label{eq:ARminimizer}
\end{equation}
Algorithmically equivalent to Hankel-DMD are linear autoregressive models~\cite{brockwell,shumway,davissparse}. The aim is generally not to find eigenvalues- and functions of the Koopman operator but simply to model the evolution of $\observ$ by a linear approximation of the mapping between $[\observ,\mathcal{K}^{-1}\circ \observ,\dots,\mathcal{K}^{-\memory+1}\circ \observ]$ and $\mathcal{K}\circ \observ$. Linear AR models suffer from the same problems as linear models in general: typically they converge to a fixed point and are unable to generate nonlinear behaviour (if not applied to nonlinear coordinate transforms, such as a Fourier transform beforehand, as in \cite{pan}). If the dynamics are periodic, linear AR models can be used to find the periodicity but this typically requires a large memory depth (as in Example~\ref{exa:KSlifting}). This again yields the need for a high number of parameters from which unstable models can ensue. As such, they are usually not recommended for the identification of nonlinear dynamics.

\paragraph{Nonlinear AR (NAR)}
The NAR problem reads,
\begin{equation}
\theta = \argmin_{\theta'\in \R^{\dimD\times L}} \Vert \begin{bmatrix}
\observ(\datapoint_\memory) & \dots & \observ(\datapoint_{T})\\
\end{bmatrix}
- \theta' \begin{bmatrix}
\observ_1(\datapoint_{\memory-1},\dots,\datapoint_0) & \dots & \observ_1(\datapoint_{T-1},\dots,\datapoint_{T-\memory})\\
\vdots & & \vdots\\
\observ_L(\datapoint_{\memory-1},\dots,\datapoint_0) & \dots & \observ_L(\datapoint_{T-1},\dots,\datapoint_{T-\memory})
\end{bmatrix}
\Vert_F.
\label{eq:NARminimizer}
\end{equation}
In nonlinear AR models, one combines nonlinear basis functions and memory by estimating the function $F$ in the model form  $\datapoint_t = F(\datapoint_{t-1},\dots,\datapoint_{t-\memory})$. Often one defines basis functions $[\observ_1,\dots,\observ_L,\mathcal{K}^{-1}\circ \observ_1,\dots,\mathcal{K}^{-\memory+1}\circ \observ_L]^T$, making these models a close variant to the ones determined in SINDy. Similarly to SINDy or nonlinear regression problems in general, these are prone to be ill-conditioned, which might cause trajectories generated with the ensuing model to be highly unstable. mSPA represents a nonlinear AR model which, by construction of the basis $\Psi^\memory$ and the constraint on the coefficients in $\hat{\Lambda}^\memory$, guarantees stability. 

In Example~\ref{exa:KSlifting} we applied these methods to a variant of the Kuramoto--Sivashinsky PDE. Note that system identification for the KS PDE has also been performed in \cite{linfu} and \cite{lulin2015}. There the authors use a  Fourier representation of the solution and estimate a Nonlinear Autoregressive Moving Average Model with Exogeneous Input (NARMAX) model on some of the Fourier modes. NARMAX models are essentially nonlinear AR models which further include moving average terms of stochastic inputs and a control term (the exogeneous input). For more details we refer to~\cite{linfu,lulin2015,luchorin2015,billings}. The authors produce good results with certain memory depths but admit that they do not always derive stable models (including unstable linear AR models). Additionally, they use a Fourier perspective which is not necessarily sensible for not wave-like evolving systems. We do not see the method deployed in \cite{linfu,lulin2015} and mSPA as competition but rather complementary to each other, since NARMAX estimates memory-based dynamics of a manually chosen observable with chosen basis functions and mSPA estimates stable dynamics on the barycentric coordinates of points on the polytope derived with \ref{eq:SPA1} with specific basis functions and the option to lift back into the original state space. Future work could encompass the combination of the two approaches by, e.g., augmentation of mSPA to models with exogenous input, i.e., with control terms.

\small
\bibliographystyle{myalpha}
\bibliography{memorySPA_Paper}

\newcommand{\etalchar}[1]{$^{#1}$}
\begin{thebibliography}{WDKW21}

\bibitem[AM16]{arbabi}
H.~Arbabi and I.~Mezic.
\newblock Ergodic theory, {D}ynamic {M}ode {D}ecomposition, and computation of
  spectral properties of the {K}oopman operator.
\newblock {\em SIAM Journal on Applied Dynamical Systems}, 16:2096--2126, 11
  2016.

\bibitem[BBK{\etalchar{+}}21]{BBKLSH21}
P.~Bevanda, M.~Beier, S.~Kerz, A.~Lederer, S.~Sosnowski, and S.~Hirche.
\newblock Koopmanizing{F}lows: Diffeomorphically learning stable {K}oopman
  operators.
\newblock {\em Preprint, arXiv: 2112.04085}, 2021.

\bibitem[BBM20]{bernetti}
M.~Bernetti, M.~Bertazzo, and M.~Masetti.
\newblock Data-driven molecular dynamics: A multifaceted challenge.
\newblock {\em Pharmaceuticals}, 13(9), 2020.

\bibitem[BBP{\etalchar{+}}17]{havok}
S.~Brunton, B.~Brunton, J.~Proctor, E.~Kaiser, and J.~Kutz.
\newblock Chaos as an intermittently forced linear system.
\newblock {\em Nature Communications}, 8, 19, 08 2017.

\bibitem[BD91]{brockwell}
P.~J. Brockwell and R.~A. Davis.
\newblock {\em Time {S}eries: Theory and {M}ethods}, volume~2.
\newblock Springer, 1991.

\bibitem[Bil13]{billings}
S.~Billings.
\newblock {\em Nonlinear {S}ystem {I}dentification: {NARMAX} {M}ethods in the
  {T}ime, {F}requency, and {S}patio-{T}emporal {D}omains}, volume~1.
\newblock Wiley, 2013.

\bibitem[BK19]{kutzNNbook}
S.~Brunton and J.~Kutz.
\newblock {\em Neural Networks and Deep Learning}, pages 195--226.
\newblock 02 2019.

\bibitem[BKK{\etalchar{+}}17]{bittracher}
A.~Bittracher, P.~Koltai, S.~Klus, R.~Banisch, M.~Dellnitz, and C.~Schütte.
\newblock Transition manifolds of complex metastable systems: Theory and
  data-driven computation of effective dynamics.
\newblock {\em Journal of Nonlinear Science}, 28, 04 2017.

\bibitem[BMR00]{birgin}
E.~G. Birgin, J.~M. Martínez, and M.~Raydan.
\newblock Nonmonotone spectral projected gradient methods on convex sets.
\newblock {\em SIAM Journal on Optimization}, pages 1196--1211, 2000.

\bibitem[BPK16a]{sindy}
S.~L. Brunton, J.~P.~L. Proctor, and J.~N. Kutz.
\newblock Discovering governing equations from data by sparse identification of
  nonlinear dynamical systems.
\newblock {\em Proceedings of the National Academy of Sciences},
  113(15):3932--3937, 2016.

\bibitem[BPK16b]{sindyc}
S.~L. Brunton, J.~P.~L. Proctor, and J.~N. Kutz.
\newblock Sparse {I}dentification of {N}onlinear {D}ynamics with {C}ontrol
  ({SINDYc}).
\newblock {\em IFAC-PapersOnLine Issue 18}, 49:710--715, 2016.

\bibitem[CB94]{cutler}
A.~Cutler and L.~Breiman.
\newblock Archetypal analysis.
\newblock {\em Technometrics}, 36(4):338--347, 1994.

\bibitem[CCV09]{cencini}
M.~Cencini, F.~Cecconi, and A.~Vulpiani.
\newblock {\em Chaos: From Simple Models to Complex Systems}.
\newblock World Scientific, 2009.

\bibitem[CKM86]{chua}
L.~Chua, M.~Komuro, and T.~Matsumoto.
\newblock The double scroll family.
\newblock {\em IEEE Transactions on Circuits and Systems}, 33(11):1072--1118,
  1986.

\bibitem[CL15]{luchorin2015}
A.~Chorin and F.~Lu.
\newblock Discrete approach to stochastic parametrization and dimension
  reduction in nonlinear dynamics.
\newblock {\em Proceedings of the National Academy of Sciences of the United
  States of America}, 112, 07 2015.

\bibitem[CLKB19]{champion}
K.~Champion, B.~Lusch, J.~Kutz, and S.~Brunton.
\newblock Data-driven discovery of coordinates and governing equations.
\newblock {\em Proceedings of the National Academy of Sciences},
  116:22445–22451, 10 2019.

\bibitem[CM02]{coxETDRK4}
S.~Cox and P.~Matthews.
\newblock Exponential time differencing for stiff systems.
\newblock {\em Journal of Computational Physics}, 176:430--455, 03 2002.

\bibitem[DHvMZ15]{dellnitz}
M.~Dellnitz, M.~Hessel-von Molo, and A.~Ziessler.
\newblock On the computation of attractors for delay differential equations.
\newblock {\em Journal of Computational Dynamics}, 3, 08 2015.

\bibitem[DW05]{weberpcca}
P.~Deuflhard and M.~Weber.
\newblock Robust {P}erron cluster analysis in conformation dynamics.
\newblock {\em Linear Algebra and its Applications}, 398:161--184, 2005.
\newblock Special Issue on Matrices and Mathematical Biology.

\bibitem[DZZ12]{davissparse}
R.~Davis, P.~Zang, and T.~Zheng.
\newblock Sparse vector autoregressive modeling.
\newblock {\em Journal of Computational and Graphical Statistics}, 30, 2012.

\bibitem[EFHN15]{eisner2015operator}
T.~Eisner, B.~Farkas, M.~Haase, and R.~Nagel.
\newblock {\em Operator theoretic aspects of ergodic theory}, volume 272.
\newblock Springer, 2015.

\bibitem[FWM10]{freund}
R.~J. Freund, W.~J. Wilson, and D.~L. Mohr.
\newblock {\em Statistical Methods, Chapter 13 - Special Types of Regression}.
\newblock Academic Press, Boston, third edition edition, 2010.

\bibitem[FYR{\etalchar{+}}21]{fan2021learning}
F.~Fan, B.~Yi, D.~Rye, G.~Shi, and I.~R. Manchester.
\newblock Learning stable {K}oopman embeddings.
\newblock {\em arXiv preprint arXiv:2110.06509}, 2021.

\bibitem[GCS{\etalchar{+}}13]{rubin}
A.~Gelman, J.~B. Carlin, H.~S. Stern, D.~B. Dunson, A.~Vehtari, and D.~B.
  Rubin.
\newblock Bayesian data analysis.
\newblock {\em CRC Press, Texts in Statistical Science}, 3, 2013.

\bibitem[GGH21]{gilani}
F.~Gilani, D.~Giannakis, and J.~Harlim.
\newblock Kernel-based prediction of non-{M}arkovian time series.
\newblock {\em Physica D: Nonlinear Phenomena}, 418:132829, 01 2021.

\bibitem[GKES19]{gelss}
P.~Gelß, S.~Klus, J.~Eisert, and C.~Schütte.
\newblock Multidimensional approximation of nonlinear dynamical systems.
\newblock {\em Journal of Computational and Nonlinear Dynamics}, 14, 03 2019.

\bibitem[GONH18]{olsson}
S.~Gerber, S.~Olsson, F.~Noé, and I.~Horenko.
\newblock A scalable approach to the computation of invariant measures for
  high-dimensional {M}arkovian systems.
\newblock {\em Scientific Reports}, 8, 01 2018.

\bibitem[GPNH20]{spaPaper}
S.~Gerber, L.~Pospisil, M.~Navandar, and I.~Horenko.
\newblock Low-cost scalable discretization, prediction, and feature selection
  for complex systems.
\newblock {\em Science Advances}, 6(5), 2020.

\bibitem[Gue13]{guessab}
A.~Guessab.
\newblock Generalized barycentric coordinates and approximations of convex
  functions on arbitrary convex polytopes.
\newblock {\em Computers \& Mathematics with Applications}, 66(6):1120--1136,
  2013.

\bibitem[HO20]{hamzi}
B.~Hamzi and H.~Owhadi.
\newblock Learning dynamical systems from data: a simple cross-validation
  perspective.
\newblock {\em Physica D Nonlinear Phenomena}, 07 2020.

\bibitem[Hor20]{eSPAIllia}
I.~Horenko.
\newblock On a scalable entropic breaching of the overfitting barrier for small
  data problems in machine learning.
\newblock {\em Neural Computation}, 32:1--17, 06 2020.

\bibitem[HP18]{husic}
B.~Husic and V.~Pande.
\newblock Markov {S}tate {M}odels: From an art to a science.
\newblock {\em Journal of the American Chemical Society}, 140
  (7):2386--–2396, 01 2018.

\bibitem[Huk93]{huke}
J.~Huke.
\newblock Embedding nonlinear dynamical systems, a guide to {T}akens theorem.
\newblock {\em Internal Report, DRA Malvern}, 1993.

\bibitem[JMAO21]{xunbi}
X.~Ji, T.~Molnár, S.~Avedisov, and G.~Orosz.
\newblock Learning the dynamics of time delay systems with trainable delays.
\newblock In {\em Proceedings of the 3rd Conference on Learning for Dynamics
  and Control}, volume 144 of {\em Proceedings of Machine Learning Research},
  pages 930--942. PMLR, 07 -- 08 June 2021.

\bibitem[KKS16]{KlKoSch16}
S.~Klus, P.~Koltai, and C.~Sch{\"u}tte.
\newblock On the numerical approximation of the {P}erron--{F}robenius and
  {K}oopman operator.
\newblock {\em Journal of Computational Dynamics}, 3(1):51--79, 2016.

\bibitem[KLL{\etalchar{+}}21]{stuart}
N.~Kovachki, Z.~Li, B.~Liu, K.~Azizzadenesheli, K.~Bhattacharya, A.~Stuart, and
  A.~Anandkumar.
\newblock Neural operator: Learning maps between function spaces, 08 2021.

\bibitem[KNK{\etalchar{+}}18]{klusKoopman}
S.~Klus, F.~Nüske, P.~Koltai, H.~Wu, I.~Kevrekidis, C.~Schütte, and F.~Noé.
\newblock Data-driven model reduction and transfer operator approximation.
\newblock {\em Journal of Nonlinear Science}, 28, 06 2018.

\bibitem[KNP{\etalchar{+}}20]{klus2020}
S.~Klus, F.~Nüske, S.~Peitz, J.-H. Niemann, C.~Clementi, and C.~Schütte.
\newblock Data-driven approximation of the {K}oopman generator: Model
  reduction, system identification, and control.
\newblock {\em Physica D: Nonlinear Phenomena}, 406:132416, 02 2020.

\bibitem[Kur78]{kuramoto}
Y.~Kuramoto.
\newblock {Diffusion-Induced Chaos in Reaction Systems}.
\newblock {\em Progress of Theoretical Physics Supplement}, 64:346--367, 02
  1978.

\bibitem[LBB{\etalchar{+}}11]{lane}
T.~Lane, G.~Bowman, K.~Beauchamp, V.~Voelz, and V.~Pande.
\newblock Markov {S}tate {M}odel reveals folding and functional dynamics in
  ultra-long {MD} trajectories.
\newblock {\em Journal of the American Chemical Society}, 133:18413--9, 11
  2011.

\bibitem[LDBK17]{liEDMD}
Q.~Li, F.~Dietrich, E.~M. Bollt, and I.~G. Kevrekidis.
\newblock Extended {D}ynamic {M}ode {D}ecomposition with dictionary learning: A
  data-driven adaptive spectral decomposition of the {K}oopman operator.
\newblock {\em Chaos: An Interdisciplinary Journal of Nonlinear Science},
  27(10):103111, 2017.

\bibitem[LKB18]{lusch}
B.~Lusch, J.~Kutz, and S.~Brunton.
\newblock Deep learning for universal linear embeddings of nonlinear dynamics.
\newblock {\em Nature Communications}, 9:1--10, 11 2018.

\bibitem[LL21]{linfu}
K.~K. Lin and F.~Lu.
\newblock Data-driven model reduction, {W}iener projections, and the
  {M}ori-{Z}wanzig formalism.
\newblock {\em Journal of Computational Physics}, 424:109864, 2021.

\bibitem[LLC15]{lulin2015}
F.~Lu, K.~Lin, and A.~Chorin.
\newblock Data-based stochastic model reduction for the
  {K}uramoto--{S}ivashinsky equation.
\newblock {\em Physica D: Nonlinear Phenomena}, 340, 09 2015.

\bibitem[Mey00]{meyer}
C.~D. Meyer.
\newblock {\em Matrix Analysis and Applied Linear Algebra}.
\newblock Society for Industrial and Applied Mathematics, USA, 2000.

\bibitem[MSVS19]{marquez}
B.~A. Marquez, J.~Suarez-Vargas, and B.~J. Shastri.
\newblock Takens-inspired neuromorphic processor: A downsizing tool for random
  recurrent neural networks via feature extraction.
\newblock {\em Phys. Rev. Research}, 1:033030, Oct 2019.

\bibitem[PD20]{pan}
S.~Pan and K.~Duraisamy.
\newblock On the structure of time-delay embedding in linear models of
  non-linear dynamical systems.
\newblock {\em Chaos: An Interdisciplinary Journal of Nonlinear Science},
  30:073135, 07 2020.

\bibitem[Rob05]{robinson}
J.~Robinson.
\newblock A topological delay embedding theorem for infinite-dimensional
  dynamical systems.
\newblock {\em Nonlinearity}, 18:2135, 07 2005.

\bibitem[SHD01]{huisinga}
C.~Schütte, W.~Huisinga, and P.~Deuflhard.
\newblock Transfer operator approach to conformational dynamics in biomolecular
  systems.
\newblock {\em Ergodic theory, analysis, and efficient simulation of dynamical
  systems}, pages 191--223, 2001.

\bibitem[Siv80]{sivashinsky}
G.~I. Sivashinsky.
\newblock On flame propagation under conditions of stoichiometry.
\newblock {\em SIAM Journal on Applied Mathematics}, 39(1):67--82, 1980.

\bibitem[SS11]{shumway}
R.~Shumway and D.~Stoffer.
\newblock {\em Time Series Analysis and Its Applications With R Examples},
  volume~9.
\newblock 01 2011.

\bibitem[SSAB20]{scherer}
C.~Scherer, R.~Scheid, D.~Andrienko, and T.~Bereau.
\newblock Kernel-based machine learning for efficient simulations of molecular
  liquids.
\newblock {\em Journal of Chemical Theory and Computation}, 16:3194--3204, 04
  2020.

\bibitem[SYC91]{sauer}
T.~Sauer, J.~Yorke, and M.~Casdagli.
\newblock Embedology.
\newblock 65:579–616, 1991.

\bibitem[Tak81]{takens}
F.~Takens.
\newblock Detecting strange attractors in turbulence.
\newblock 898:366--381, 1981.

\bibitem[TRL{\etalchar{+}}14]{dmd}
J.~H. Tu, C.~W. Rowley, D.~M. Luchtenburg, S.~L. Brunton, and J.~N. Kutz.
\newblock On {D}ynamic {M}ode {D}ecomposition: Theory and applications.
\newblock {\em Journal of {C}omputational {D}ynamics}, 1:391--421, 2014.

\bibitem[TZ19]{teng}
Q.~Teng and L.~Zhang.
\newblock Data driven nonlinear dynamical systems identification using
  multi-step {CLDNN}.
\newblock {\em AIP Advances}, 9(8):085311, 2019.

\bibitem[VBW{\etalchar{+}}18]{vlachas}
P.~Vlachas, W.~Byeon, Z.~Wan, T.~Sapsis, and P.~Koumoutsakos.
\newblock Data-driven forecasting of high-dimensional chaotic systems with
  {L}ong {S}hort-{T}erm {M}emory networks.
\newblock {\em Proceedings of the Royal Society A: Mathematical, Physical and
  Engineering Science}, 474, 02 2018.

\bibitem[WDKW21]{wunderling}
N.~Wunderling, J.~F. Donges, J.~Kurths, and R.~Winkelmann.
\newblock Interacting tipping elements increase risk of climate domino effects
  under global warming.
\newblock {\em Earth System Dynamics}, 12(2):601--619, 2021.

\bibitem[WKR14]{williams}
M.~Williams, I.~Kevrekidis, and C.~Rowley.
\newblock A data-driven approximation of the {K}oopman operator: Extending
  {D}ynamic {M}ode {D}ecomposition.
\newblock {\em Journal of Nonlinear Science}, 25, 2014.

\bibitem[Wu87]{chuaWu}
S.~Wu.
\newblock Chua's circuit family.
\newblock {\em Proceedings of the IEEE}, 75(8):1022--1032, 1987.

\end{thebibliography}
\end{document}